\documentclass[a4paper, 12pt]{scrartcl}

\usepackage[numberTheoremsWithin=section, numberEquationsWithinTheorems]{mathEnv}

\usepackage[bibencoding=ascii, style=alphabetic]{biblatex}
\usepackage[]{main_arxiv}
\newcommand{\R}{\ensuremath{\mathbb{R}}}
\newcommand{\C}{\ensuremath{\mathbb{C}}}
\newcommand{\N}{\ensuremath{\mathbb{N}}}
\newcommand{\K}{\ensuremath{\mathbb{K}}}
\newcommand{\Z}{\ensuremath{\mathbb{Z}}}
\newcommand{\Q}{\ensuremath{\mathbb{Q}}}

\newcommand{\eps}{\ensuremath{\varepsilon}}

\DeclareMathOperator{\idSymb}{\mathrm{id}}
\newcommand{\id}[1]{\idSymb_{#1}}

\newcommand{\sub}{\ensuremath{\subseteq}}
\newcommand{\set}[2]{\ensuremath{\{#1 : #2\} }}
\newcommand{\sset}[1]{\ensuremath{\{#1\}}}
\newcommand{\ndef}{\colonequals}
\newcommand{\defn}{\equalscolon}
\newcommand{\rest}[3][]{\ensuremath{#2|_{#3}\ifthenelse{ \equal{#1}{} }{}{^{#1}}} }

\DeclareMathOperator{\distop}{\mathrm{dist}}
\newcommand{\dist}[2]{\ensuremath \distop(#1,#2)}
\newcommand{\abs}[1]{\ensuremath \lvert #1\rvert}
\newcommand{\norm}[1]{\ensuremath \lVert #1\rVert}
\newcommand{\MetaBall}[4][]
{
	\ensuremath
	{
	\ifthenelse{ \equal{#1}{} }
	{
		 #4_{#3}(#2)
	}
	{
		#4_{#1}(#2, #3)
	}
	}
}
\newcommand{\Ball}[3][]{\MetaBall[#1]{#2}{#3}{B}}
\newcommand{\clBall}[3][]{ \MetaBall[#1]{#2}{#3}{\cl{B}} }

\newcommand{\neigh}[2][]{\ensuremath{ \mathcal{U}\ifthenelse{ \equal{#1}{} }{}{_{#1}}(#2) }}
\newcommand{\neighO}[2][]{\ensuremath{ \mathcal{U}\ifthenelse{ \equal{#1}{} }{}{_{#1}}^\circ(#2) }}
\newcommand{\cl}[1]{\overline{#1}}

\newcommand{\fami}[3]{\ensuremath{(#1)_{#2 \in #3}}}

\newcommand{\famiI}[1]{\fami{#1}{i}{I}}

\newcommand{\dotcup}{\ensuremath{\mathaccent\cdot\cup}}
\newcommand{\disjointU}[2][]{\dotcup\ifthenelse{\equal{#1}{}}{}{_{#1}}#2}
\newcommand{\disjointUI}[1]{\disjointU[i\in I]{#1}}

\newcommand{\Lip}[3]{\ensuremath{\mathrm{Lip}_{#1}^{#2}(#3)}}
\newcommand{\Rint}[5][\int]{\ensuremath {#1}_{#2}^{#3} #4\, d #5}
\newcommand{\Mint}[3][\int]{\Rint[#1]{0}{1}{#2}{#3}}

\newcommand{\dA}[4][]{
	\ensuremath{
	d\ifthenelse{ \equal{#1}{} }{}{^{(#1)}} 
		#2 \ifthenelse{\equal{#3}{} \and \equal{#4}{} }{}{(#3;#4)}
		}
}
\newcommand{\FAbl}[2][]{
	\ensuremath{
	D\ifthenelse{ \equal{#1}{} }{}{^{(#1)}}#2
	}
}
\newcommand{\parFAbl}[3][]{ \ensuremath{ \FAbl[#1]{_{#2} #3} } }

\newcommand{\AblCurve}[2][0]{\ensuremath{\frac{d}{d #2}|_{#2 = #1}}}

\newcommand{\LLA}[2][]{ \ensuremath{\delta_{\ell}\ifthenelse{\equal{#2}{} }{}{(#2)} } }
\newcommand{\RLA}[2][]{ \ensuremath{\delta_{\rho}\ifthenelse{\equal{#2}{} }{}{(#2)} } }

\newcommand{\Tang}[2][]{
\ensuremath{\mathbf{T}\ifthenelse{ \equal{#1}{} }{}{_{#1} }
	\ifthenelse{ \equal{#2}{} }%
	{}%
	{#2}%
	}
}

\newcommand{\ConDiff}[4][]
{
	\ensuremath{
	\mathcal{C}^{#4}_{#1}
	\ifthenelse{ \equal{#2}{} \and \equal{#3}{} }
	{}
	{(#2, #3)}
	}
}

\newcommand{\FC}[4][]
{
	\ensuremath{
	\mathcal{FC}^{#4}
	\ifthenelse{ \equal{#2}{} \and \equal{#3}{} }
	{}
	{(#2, #3)}
	}
}

\newcommand{\GewFunk}{\cW}
\newcommand{\hn}[3]{\ensuremath\norm{#1}_{#2, #3}}

\newcommand{\CF}[4]{\ensuremath\mathcal{C}^{#4}_{#3}(#1, #2)}
\newcommand{\CFo}[4]{\CF{#1}{#2}{#3}{\partial, #4}}

\newcommand{\CcF}[3]{\CF{#1}{#2}{\GewFunk}{#3}}

\newcommand{\CcFo}[3]{\CFo{#1}{#2}{\GewFunk}{#3}}

\newcommand{\CFpro}[6]{\CF{#1}{#2}{#3}{#4}_{#5 \in #6}}
\newcommand{\CFproI}[4]{\CFpro{#1}{#2}{#3}{#4}{i}{I}}

\newcommand{\CcFproI}[3]{\CFpro{#1}{#2}{\GewFunk}{#3}{i}{I}}
\newcommand{\CFfoPRO}[7]{\CFpro{#1}{#2}{#3}{{#5}_\partial, #4}{#6}{#7}}

\newcommand{\CcFfoPROi}[4]{\CFfoPRO{#1}{#2}{\GewFunk}{#3}{#4}{i}{I}}

\newcommand{\CcFoPRO}[5]{\CcFo{#1}{#2}{#3}_{#4 \in #5}}
\newcommand{\CcFoPROi}[3]{\CcFoPRO{#1}{#2}{#3}{i}{I}}

\newcommand{\BC}[3]{\ensuremath{\mathcal{BC}^{#3}(#1, #2)}}

\newcommand{\compIdRaw}{\tilde{\mathfrak{c}}}

\newcommand{\compIdDiffKLWeights}[4]{\mathfrak{c}_{#4, #3}^{#1, #2}}
\newcommand{\compIdSmoothLWeights}[3]{\mathfrak{c}_{#3}^{#1, #2}}
\newcommand{\compIdWeights}[2]{\mathfrak{c}_{#2}^{#1}}

\newcommand{\compIdDiffKLcW}[3]{\compIdDiffKLWeights{#1}{#2}{#3}{\GewFunk}}
\newcommand{\compIdSmoothLcW}[2]{\compIdSmoothLWeights{#1}{#2}{\GewFunk}}
\newcommand{\compIdcW}[1]{\compIdWeights{#1}{\GewFunk}}

\newcommand{\InvIdWeights}[2]{I^{#1}_{#2}}
\newcommand{\InvIdcW}[1]{\InvIdWeights{#1}{\GewFunk}}

\newcommand{\cW}{\ensuremath{\mathcal{W}}}

\newcommand{\ExtWeights}[1]{\ensuremath{#1_{\mathrm{max}}}}

\newcommand{\Opnorm}[1]{\norm{#1}_{op}}

\DeclareMathOperator{\Id}{\mathrm{Id}}
\newcommand{\idco}{\ensuremath{\Id}}
\newcommand{\MaMu}{\ensuremath{\cdot}}
\newcommand{\eval}{\ensuremath{\cdot}}

\newcommand{\QuasiInv}{\ensuremath{QI}}

\newcommand{\Lin}[3][]{
\ensuremath{
\ifthenelse{\equal{#1}{}}
	{
		\ifthenelse{\equal{#2}{#3}}
			{\mathrm{L}(#2)}
			{\mathrm{L}(#2,#3)}
	}
	{
		\mathrm{L}^{#1}(#2, #3)
	}
}
}
\newcommand{\EmbA}[3][]{\ensuremath{\mathcal E}_{#2,#3}\ifthenelse{\equal{#1}{}}{}{^{#1}}}

\newcommand{\Evol}[3]{\ensuremath{\mathrm{Evol}_{#1}^{#3}\ifthenelse{\equal{#2}{}}{}{(#2)}}}
\newcommand{\evol}[3]{\ensuremath{\mathrm{evol}_{#1}^{#3}\ifthenelse{\equal{#2}{}}{}{(#2)}}}

\newcommand{\SX}{\ensuremath X}
\newcommand{\SY}{\ensuremath Y}
\newcommand{\SZ}{\ensuremath Z}

\newcommand{\UF}{\ensuremath U}
\newcommand{\VF}{\ensuremath V}
\newcommand{\WF}{\ensuremath W}

\newcommand{\G}{G}

\newcommand{\normsOn}[1]{\ensuremath{\mathcal{N}(#1)}}

\newcommand{\grenzExp}[3][]{\ensuremath{R^{E\ifthenelse{\equal{#1}{}}{}{, #1}}_{#2, #3}}}
\newcommand{\grenzLog}[3][]{\ensuremath{R^{L\ifthenelse{\equal{#1}{}}{}{, #1}}_{#2, #3}}}

\newcommand{\BndFstAblRex}[3][]{\ensuremath{C^{E, (1)}_{#2, #3\ifthenelse{\equal{#1}{}}{}{, #1}}}}
\newcommand{\BndSndAblRex}[3][]{\ensuremath{C^{E, 2}_{#2, #3\ifthenelse{\equal{#1}{}}{}{, #1}}}}
\newcommand{\BndFstAblRlog}[3][]{\ensuremath{C^{L, (1)}_{#2, #3\ifthenelse{\equal{#1}{}}{}{, #1}}}}
\newcommand{\BndSndAblRlog}[3][]{\ensuremath{C^{L, 2}_{#2, #3\ifthenelse{\equal{#1}{}}{}{, #1}}}}

\newcommand{\VectFLok}[2]{\ensuremath{#1_{#2}}}

\newcommand{\Linf}[2][J]{\ensuremath{\ell^\infty_{#1}(#2)}}

\newcommand{\ExpIdWeights}[4][]{\ensuremath{E_{#2, #3}^{#4\ifthenelse{\equal{#1}{}}{}{,#1}} }}

\newcommand{\LogIdWeights}[4][]{\ensuremath{L_{#2, #3}^{#4\ifthenelse{\equal{#1}{}}{}{,#1}} }}

\newcommand{\RadExpFibInv}[3][]{\ensuremath{R_{#2, #3}\ifthenelse{\equal{#1}{}}{}{^{#1}} }}

\begin{document}

\title{Differentiable mappings between weighted restricted products}
\author{Boris Walter}
\date{}
\publishers{
\small{}
Universität Paderborn\\ Institut für Mathematik\\
Warburger Straße 100\\ 33098 Paderborn\\
E-Mail: bwalter@math.upb.de
}
\maketitle

{\small
\begin{abstract}
\noindent{}In this paper, we introduce restricted products for families of locally convex spaces
and formulate criteria ensuring that mappings into such products are continuous or smooth.
As a special case, can define restricted products of weighted function spaces
and obtain results concerning continuity and differentiability
properties of natural non-linear mappings between such spaces.
These concepts and results are the basis for the study of weighted vector fields
on Riemannian manifolds in a subsequent work
(see [B. Walter, \emph{Weighted diffeomorphism groups
of Riemannian manifolds}, {\ttfamily arXiv: \href{http://arxiv.org/abs/1601.02834}{1601.02834}}]),
which serve as modelling spaces for suitable infinite-dimensional Lie groups of diffeomorphisms.
\end{abstract}

}

\newcommand{\Disk}{\ensuremath{\mathbb{D}}}

\newrefformat{chp}{\hyperref[{#1}]{Chapter~\ref*{#1}}}
\newrefformat{sec}{\hyperref[{#1}]{Section~\ref*{#1}}}
\newrefformat{susec}{\hyperref[{#1}]{Subsection~\ref*{#1}}}
\newrefformat{sususec}{\hyperref[{#1}]{Subsection~\ref*{#1}}}
\newrefformat{app}{\hyperref[{#1}]{Appendix~\ref*{#1}}}

\newrefformat{def}{\hyperref[{#1}]{Definition~\ref*{#1}}}

\newrefformat{ivp}{\hyperref[{#1}]{initial value problem~\ref*{#1}}}
\newrefformat{est}{\hyperref[{#1}]{estimate~\ref*{#1}}}
\newrefformat{id}{\hyperref[{#1}]{identity~\ref*{#1}}}
\newrefformat{bed}{\hyperref[{#1}]{condition~\ref*{#1}}}
\newrefformat{char}{\hyperref[{#1}]{characterization~\ref*{#1}}}

\newrefformat{enum1}{\ref{#1}}

\section{Introduction}
In the book \cite{MR2952176}, Lie groups of weighted diffeomorphisms on Banach spaces were constructed
(see already \cite{MR2263211} for rapidly decreasing diffeomorphisms
of the real line; cf. \cite{MR3132089} and \cite{MR3313140} for later developments).
The model space used for these groups are weighted mappings between Banach spaces.
In order to construct Lie groups of weighted diffeomorphisms on non-compact manifolds,
we need to define spaces of weighted vector fields.
The purpose of this paper, whose content is a part of the author's dissertation,
is to develop a framework for such spaces and tools to handle them efficiently.
In particular, we define and examine some kind of \emph{simultaneously} weighted functions.

As a motivating example, consider the direct product
\[
	M \ndef \R \times {\mathbb S}
\]
of the real line and the circle group. Then smooth vector fields
on $M$ can be identified with smooth functions
\[
	\gamma : \R \times \R \to \R^2
	: (x,y) \mapsto \gamma(x,y)
\]
which are $2\pi$-periodic in the $y$-variable.
To control the asymptotic behaviour of vector fields
(and the diffeomomorphisms arising from their flows)
at infinity,
it is natural to impose that $\gamma$ (and its partial derivatives)
decays polynomially as $x \to \pm \infty$ in the sense that, for each $n\in \N$,
\[
	x^n \gamma(x,y)
\]
is bounded for $(x,y)\in \R^2$ (and hence $\gamma$ tends to $0$ as $x\to \infty$).
Likewise, we could impose that $\gamma$ and all its partial derivatives
are bounded, or have exponential decay.
The preceding approach hinges on the very specific situation considered;
namely, that we have the local diffeomorphism $q\colon (t,s)\mapsto (t,e^{is})$ from $\R^2$
(on which vector fields can be identified with smooth functions $\R^2 \to \R^2$)
onto $\R\times {\mathbb S}$.
Of course, one would like to be able to describe weighted vector fields as just encountered
also without reference to $q$, and for general manifolds,
none of whose covering manifolds may admit a global chart.

To explain the basic idea of such a construction,
let $M$ be a manifold, $f : M \to \R$ a weight on $M$
and $X : M \to \Tang{M}$ a vector field.
There is no canonical way to express what it means that $X$ is bounded with respect to $f$.
In contrast, for a chart $\kappa$ for $M$
we perfectly understand what it means if the function $\VectFLok{X}{\kappa} = \dA{\kappa}{}{} \circ X \circ \kappa^{-1}$
is bounded with respect to the weight $f \circ \kappa^{-1}$.
So we may say that $X$ is bounded with respect to $f$
if all its localizations (with respect to an atlas $\mathcal{A}$) are so,
and define seminorms with respect to $f$ and an order of differentiation.
For a nonempty set $\cW \sub \R^M$ of weights, this leads to the definition of a topology on a subset of the product
$\prod_{\kappa \in \mathcal{A}}  \CF{U_\kappa}{\R^d}{\cW_\kappa}{\infty}$,
where $\cW_\kappa \ndef \set{f \circ \kappa^{-1}}{f \in \cW}$,
that generally is finer than the ordinary product topology.

It is efficient to follow an even more general approach.
First, we define a \emph{restricted product} for a family of locally convex spaces
when there exists a set $J$ such that each space has a set of generating seminorms that can be indexed over $J$,
and prove some results about these kind of spaces.
After that, we define \emph{weighted restricted products}.
These consist of functions that are defined on the disjoint union of open subsets of arbitrary normed spaces,
and are bounded w.r.t. weights which also are defined on this union.

Of particular interest is the question of whether operations between these spaces
that are defined factorwise are continuous or smooth.
We will see that many maps of this type behave quite well,
and their exact continuity and differentiability properties
(as recorded in Propositions \ref{prop:simultane_SP_BCinf0_Produkt}, \ref{prop:Simultane_Koor-Kompo_diffbar} and \ref{prop:Simultane_Inv-Kompo_glatt})
are the main results of this paper
and the backbone of the construction of weighted
diffeomorphism groups in \cite{Walter_Diss_p2}.

We mention that differentiable maps between
weighted sequence spaces isomorphic to $c_0(E)$ (with values in a Banach space $E$)
have also been studied by \cite{MR0271992} and \cite{MR0418164} to some extent,
and used to construct stable manifolds around hyperbolic fixed points of
time-discrete smooth dynamical systems on Banach manifolds (using Irwin's method).

Differentiable maps between locally convex direct sums of locally convex spaces
(into which spaces of compactly supported vector fields can be embedded)
were studied in \cite{MR2005280}.
They simplify the proofs for smoothness of the relevant non-linear mappings
in the construction of a Lie group structure on the diffeomorphism group
of a non-compact manifold (first treated in \cite{MR583436}),
see \cite{glockner_patched02}
and \cite{MR3328452} (where the method is extended to diffeomorphism groups of orbifolds).

The restricted products and differentiability properties
discussed in this article play an analogous role in the case of weighted diffeomorphism groups.
Our results on simultaneous superposition can be regarded as a substitute of the
familiar $\omega$-lemma for superposition on spaces of sections
(see, e.g., \cite{MR583436} or \cite{MR0248880})
in the weighted situation.
Finally, we mention that concepts of \enquote{boundedness} for vector fields
(and their covariant derivatives) can also be formulated in the context of bounded geometry,
and have been used to construct certain diffeomorphism groups in this setting
(see \cite{MR2343536}), using different methods.
\section{Definitions and previous results}
Before we start, we have to repeat some of the notation and results of \cite{MR2952176}.
We set $\cl{S} \ndef S \cup \sset{\infty}$ for $S \in \sset{\R, \N}$.
Other notation is introduced when it is first used.
\subsection{Spaces of weighted functions}
\begin{defi}
	Let $\SX$ and $\SY$ be normed spaces
	and $\UF \subseteq \SX$ an open nonempty set.
	For $k \in \N$ and a map $f: \UF \to \cl{\R}$, we define the quasinorm
	\[
		\hn{\cdot}{f}{k}
		:\FC{\UF}{\SY}{k}\to [0,\infty]
		: \phi \mapsto \sup\set{\abs{f(x)}\,\Opnorm{\FAbl[k]{\phi}(x)} }{x\in \UF}
	\]
	on the set of $k$-times Fréchet differentiable functions.
	Furthermore, for any nonempty set $\cW \sub \cl{\R}^\UF$
	and $k \in \cl{\N}$ we define the vector space
	\[
		\CF{\UF}{\SY}{\cW}{k}
		\ndef
		\{ \gamma\in\FC{\UF}{\SY}{k}:
			(\forall f\in\cW, \ell \in \N, \ell \leq k)\:
			\hn{\gamma}{f}{\ell} < \infty
		\}
	\]
	and notice that the seminorms $\hn{\cdot}{f}{\ell}$ induce a
	locally convex vector space topology on $\CF{\UF}{\SY}{\cW}{k}$.
	We call the elements of $\cW$ \emph{weights}
	and $\CF{\UF}{\SY}{\cW}{k}$ a \emph{space of weighted maps} or \emph{space of weighted functions}.

	Further, we define the \emph{maximal extension}
	$\ExtWeights{\GewFunk} \subseteq \cl{\R}^\UF$ of $\GewFunk$
	as the set of functions $f$ for which $\hn{\cdot}{f}{0}$ is a continuous seminorm
	on $\CcF{\UF}{\SY}{0}$, for each normed space $\SY$.
	Obviously $\GewFunk \subseteq \ExtWeights{\GewFunk}$
	and we can show that $\hn{\cdot}{f}{\ell}$ is a continuous seminorm on each $\CcF{\UF}{\SY}{k}$,
	provided that $f \in \ExtWeights{\GewFunk}$ and $\ell \leq k$.
\end{defi}
An important tool for dealing with higher differentiability orders is the following:
\begin{lem}[Reduction to lower order]
\label{lem:topologische_Zerlegung_von_CFk}
	Let $\SX$ and $\SY$ be normed spaces,
	$\UF \subseteq \SX$ an open nonempty set, $\GewFunk \subseteq \cl{\R}^\UF$,
	$k \in \N$ and $\gamma\in\FC{\UF}{\SY}{1}$. Then
	\[
		\gamma \in \CcF{\UF}{\SY}{k+1}
		\iff
		(\FAbl{\gamma}, \gamma)\in
		\CcF{\UF}{\Lin{\SX}{\SY}}{k} \times
			\CcF{\UF}{\SY}{0}.
	\]
	Moreover, the map
	\[
		\CcF{\UF}{\SY}{k+1}\to
		\CcF{\UF}{\Lin{\SX}{\SY}}{k} \times \CcF{\UF}{\SY}{0}
		:\gamma \mapsto (\FAbl{\gamma} , \gamma)
	\]
	is a topological embedding.
\end{lem}
Occasionally, we will need the following lemma. A more general version is stated and proved
in \cite[La.~3.4.16]{MR2952176}.
\begin{lem}\label{lem:gewichtete_Abb_Produktisomorphie-endl}
	Let $\SX$, $\SY$ and $\SZ$ be normed spaces, $\UF \subseteq \SX$ an open nonempty set,
	$k \in \cl{\N}$ and $\GewFunk \subseteq \cl{\R}^\UF$ nonempty.
	Then the map
	\[
		\CcF{\UF}{\SY \times \SZ}{k} \to \CcF{\UF}{ \SY}{k} \times \CcF{\UF}{ \SZ}{k}
		: \gamma \mapsto (\pi_{\SY} \circ \gamma, \pi_{\SZ} \circ \gamma)
	\]
	is an isomorphism of locally convex topological vector spaces.
\end{lem}
\subsection{Differentialbility and smooth maps between weighted function spaces}
We recall basic definitions for the differential calculus for maps between locally convex spaces
that is known as Kellers $C^k_c$-theory.
More information about this calculus can be found in
\cite{MR0177277}, \cite{MR0440592}, \cite{MR830252}, \cite{MR583436}, \cite{MR1911979} or \cite{MR2261066}.
\begin{defi}
	Let $\SX$ and $\SY$ be locally convex spaces,
	$\UF \subseteq \SX$ an open nonempty set
	and $f:\UF\to\SY$ a map.
	We say that $f$ is \emph{$\ConDiff{}{}{1}$} if for all $u\in\UF$ and $x\in\SX$,
	the directional derivative
	\[
		\lim_{\substack{t\to 0\\ t \neq 0}}\frac{f(u + t x) - f(u)}{t}
			\defn \dA{f}{u}{x},
	\]
	exists and the map $\dA{f}{}{}  : \UF \times \SX \to \SY$ is continuous.
	Inductively, for a $k\in\N$ we call $f$ \emph{$\ConDiff{}{}{k}$}
	if $f$ is $\ConDiff{}{}{1}$ and
	$\dA{f}{}{} :\UF\times\SX\to\SY$
	is a $\ConDiff{}{}{k - 1}$-map.
	We write $\ConDiff{\UF}{\SY}{k}$ for the set of $k$-times differentiable maps.
\end{defi}
The Continuity of parameter-dependent integrals is an useful tool
when dealing with differential quotients.
Here the integral is a weak integral; see \cite[Sec.~3]{MR2319574} for details.
In particular, the following is stated (and proved) in Prop.~3.5.
\begin{lem}[Continuity of parameter-dependent integrals]
\label{lem:Stetigkeit_parameterab_Int}
	Let $P$ be a topological space, $\SX$ a locally convex space, $I\subseteq\R$ a proper interval
	and $a, b \in I$. Further, let $f:P\times I \to \SX$ be a continuous map
	such that the weak integral
	\[
		\Rint{a}{b}{f(p,t)}{t} \defn g(p)
	\]
	exists for all $p\in P$. Then the map $g:P\to\SX$ is continuous.
\end{lem}
\subsubsection{Smooth maps between weighted function spaces}
We give two examples of smooth maps between weighted function spaces
which we will adapt to the case of weighted restricted products.
\paragraph{Composition of weighted functions}
The following result about the differentiability of composition is proved in
\cite[Sec.~4.1.1]{MR2952176}, with slightly different notation.
More precisely, the following are the assertions of La.~4.1.3 and Prop.~4.1.7.
Here, $\Disk$ denotes the unit ball of $\K \in \sset{\R, \C}$.
\begin{prop}\label{prop:Kompo_Koord_glatt}
	Let $\SX$ and $\SY$ be normed spaces, $\UF, \VF, \WF \subseteq \SX$ open nonempty subsets such that
	$\VF + \UF \subseteq \WF$ and $\VF$ is balanced,
	$\GewFunk \subseteq \cl{\R}^\WF$ with $1_\WF \in \cW$
	and $k, \ell \in \cl{\N}$.
	Then
	\[
		\compIdDiffKLcW{\SY}{ k}{ \ell}
		:\CcF{\WF}{\SY}{k + \ell + 1} \times
			\CcFo{\UF}{\VF}{k}
		\to \CcF{\UF}{\SY}{k}
		: (\gamma,\eta) \mapsto \gamma\circ(\eta + \id{\UF})
	\]
	is defined and a $\ConDiff{}{}{\ell}$-map.
	If $\ell > 0$, then it has the directional derivative
	\begin{equation}\label{id:Ableitung_Kompo}
		\dA{ \compIdDiffKLcW{\SY}{ k}{ \ell} }{\gamma, \eta}{\gamma_1, \eta_1}
		= \compIdDiffKLcW{\Lin{\SX}{\SY}}{ k}{\ell - 1}(\FAbl{\gamma}, \eta) \eval \eta_1
		+ \compIdDiffKLcW{\SY}{ k}{ \ell}(\gamma_1, \eta).
	\end{equation}
	In particular, $\compIdSmoothLcW{\SY}{k} \ndef \compIdDiffKLcW{\SY}{ k}{ \infty}$
	and $\compIdcW{\SY} \ndef \compIdDiffKLcW{\SY}{ \infty}{ \infty}$ are smooth.

	Further, for $\gamma, \gamma_0 \in \CcF{\WF}{\SY}{0} \cap \BC{\WF}{\SY}{1}$
	and suitable $\eta, \eta_0 \in \CcF{\UF}{\VF}{0}$, $f \in \GewFunk$ and $x\in\UF$
	the following estimates hold:
	\begin{equation}\label{est:Funktionswerte_Gewicht_K-Kompo}
		\abs{f(x)}\,\norm{\gamma \circ (\eta + \id{\SX})(x)}
		\leq \abs{f(x)}\, (\hn{\gamma}{1_{\sset{x} + \Disk\eta(\UF)}}{1}\,\norm{\eta(x)} + \norm{\gamma(x)})
	\end{equation}
	and
	\begin{equation}\label{est:f,0-Norm_Differenz_Kompo}
		\begin{multlined}[0.8\columnwidth]
			\hn{\compIdRaw(\gamma , \eta) -
				\compIdRaw(\gamma_0 , \eta_0)}{f}{0}
			\leq
				\hn{\gamma}{1_\WF}{1} \hn{\eta - \eta_0}{f}{0}\\
				+ \hn{\gamma - \gamma_0}{1_\WF}{1} \hn{\eta_0}{f}{0}
				+ \hn{\gamma - \gamma_0}{f}{0}
		\end{multlined}
		.
	\end{equation}
\end{prop}
\paragraph{Inversion of weighted functions}
The results about inversion in \cite[Sec.~4.2.1]{MR2952176} don't allow the treatment
of weighted functions that are defined on a subset of a vector space.
Since we encounter such functions when we are treating localized vector fields,
better tools had to be provided.
The following assertions are special cases of the more general elaborations
in \cite[Sec. 4.2.1]{arxiv_1006.5580v3}.
\begin{prop}\label{prop:Zsf_Inversion_gewAbb}
	Let $\SX$ be a Banach space, $U, V \sub \SX$ open nonempty subsets
	such that $U$ is convex and there exists $r > 0$ with $V + \Ball{0}{r} \sub U$.
	Further, let $\GewFunk \subseteq \cl{\R}^{U}$ with $1_{U} \in \GewFunk$,
	$\tau \in ]0, 1[$ and
	\[
		\mathcal{D}_\tau \ndef
		\left\set{\phi \in \CcF{U}{\SX}{\infty} }{%
			\hn{\phi}{1_U}{1} < \tau
			\text{ and }
			\hn{\phi}{1_U}{0} < \frac{r}{2} (1 - \tau)
		\right}
		.
	\]
	Then the map
	\[
		\InvIdcW{V} : \mathcal{D}_\tau \to \CcF{V}{\SX}{\infty} :
		\phi \mapsto \rest{(\phi + \id{U})^{-1}}{V} - \id{V}
	\]
	is defined and smooth. In particular, for $\phi \in \mathcal{D}_\tau$ and $\phi_1 \in \CcF{U}{\SX}{\infty}$ we have that
	\begin{equation}\label{id:Ableitung_Inversion}
		\dA{\InvIdcW{V}}{\phi}{\phi_1}
		= - \compIdcW{\SX}(\QuasiInv(\FAbl{\phi}) \eval \phi_1 + \phi_1,  \InvIdcW{V}(\phi) ).
	\end{equation}
	and
	\begin{equation}\label{id:Differential_der_inversen_Abb}
		\FAbl{\, \InvIdcW{V}(\phi)}
			= (\FAbl{\phi} \MaMu \QuasiInv(- \FAbl{\phi}) - \FAbl{\phi})
					\circ(\InvIdcW{V}(\phi) + \id{V})
		;
	\end{equation}
	here $\QuasiInv$ denotes the quasi-inversion of the algebra $\CcF{ U }{ \Lin{\SX}{\SX} }{\infty}$
	(which arises as the the superposition with $\QuasiInv_{\Lin{\SX}{\SX}}$
	and is discussed in \cite[Sec.~3.3.3.3 and App.~C]{MR2952176}). 	Further, for $\psi \in \mathcal{D}_\tau$, $f \in \GewFunk$ and $x \in V$,
	the estimates
	\begin{equation}\label{est:Abschaetzung_gewichteter_FWert_der_K-Inversion}
		\abs{f(x)}\,\norm{\InvIdcW{V}(\phi)(x)}
		\leq
		\frac{\abs{f(x)}\,\norm{\phi(x)}}{1 - \hn{\phi }{1_{U}}{1}}
	\end{equation}
	and
	\begin{equation}\label{est:f0-norm_Diff_KoorInv}
		\hn{\InvIdcW{V}(\psi) - \InvIdcW{V}(\phi)}{f}{0}
		\leq
		\tfrac{1 }{ 1 - \hn{\psi}{1_U}{1} }
		\left(
			\hn{\phi - \psi}{1_U}{1} \tfrac{\hn{\phi}{f}{0} }{ 1 - \hn{\phi}{1_U}{1}}
			+   \hn{\phi - \psi}{f}{0}
		\right).
	\end{equation}
	hold.
\end{prop}
\section{A superposition operator on weighted functions}

Before we can turn our attention to restricted products,
we examine whether a function $\Xi : U \times V\to Z$ induces a superposition operation
$\gamma \mapsto \Xi \circ (\id{U}, \gamma)$ on weighed functions.
We show that this is the case if $0\in V$, $\Xi$ maps $U \times \sset{0}$ to $0$,
and if the size of the derivatives of $\Xi$ can be covered with the weights,
see \ref{cond:est_weights_SP} for the precise phrasing.
In \refer{prop:simultane_SP_BCinf0_Produkt},
we will adapt this result to weighted restricted products.

In \cite[La. 6.2.14]{MR2952176},
a similar result was proved, but for a very different sort of weighted function space.
In contrast to assertions about superposition operators in \cite{MR2952176},
we use a more quantitative approach.

\subsection{Estimates for higher derivatives}
We give estimates for the higher derivatives of a function of two variables,
provided it is linear in its second argument.
We also turn to more special cases of such functions.
\begin{lem}	Let $\SX$, $\SY$ and $\SZ$ be normed spaces,
	$\UF \subseteq \SX$ an open nonempty set,
	$k \in \cl{\N}^*$ and $\Xi \in \FC{\UF \times \SY}{\SZ}{k}$
	a map that is linear in its second argument.
	Further, let $\ell \in \N$ with $\ell \leq k$, $x \in \UF$ and $y \in \SY$.
	\begin{assertions}
		\item\label{enum1:Erstes_par_Diff-Abb_linear_2Arg}
		The map $\parFAbl[\ell]{1}{\Xi}$
		is linear in the second argument.
		Hence $ \parFAbl[\ell]{1}{\Xi}(\UF \times \sset{0}) = \sset{0} $ and (if $\ell < k$)
		\[
			\label{id:Ableitung_Abb_linear_2Arg}
			\tag{\ensuremath{\dagger}}
			\AblCurve{t} \parFAbl[\ell]{1}{\Xi}(x + t h_1, y + t h_2)
			= \parFAbl[\ell]{1}{\Xi}(x,  h_2) + \parFAbl[\ell + 1]{1}{\Xi}(x, y) \neg h_1.
		\]
		Here, for an $(m + 1)$-linear map $b : E_1 \times \dotsb \times E_{m + 1} \to F$, for $h \in E_{m + 1}$
		we let $b \neg h$ denote the $m$-linear map
		$E_1 \times \dotsb \times E_m \to F : (x_1, \dotsc, x_m) \mapsto b(x_1, \dotsc, x_m, h)$.

		\item\label{enum1:Formel-est_hohe_Diff-Abb_linear_2Arg}
		Suppose that $\ell \geq 1$.
		Let $h^1,\dotsc, h^\ell \in \SX \times \SY$ with $h^j = (h^j_1, h^j_2)$. Then the identity
		\[
			\FAbl[\ell]{\Xi}(x, y)\eval(h^1,\dotsc, h^\ell)
			= \parFAbl[\ell]{1}{\Xi}(x, y)\eval(h^1_1,\dotsc, h^\ell_1)
			+ \sum_{j = 1}^\ell \parFAbl[\ell - 1]{1}{\Xi}(x, h^j_2)\eval \widehat{h_1^j}
		\]
		holds, where $\widehat{h_1^j} \ndef (h_1^1, \dotsc, h_1^{j - 1}, h_1^{j + 1}, \dotsc, h_1^\ell)$.
		In particular,
		\[
			\label{est:norm_l-te_Ableitung-Abb_linear_2Arg}
			\tag{\ensuremath{\dagger\dagger}}
			\Opnorm{\FAbl[\ell]{\Xi}(x, y)}
			\leq \ell \Opnorm{\parFAbl[\ell - 1]{1}{\Xi}(x, \cdot)}
			+ \Opnorm{\parFAbl[\ell]{1}{\Xi}(x, \cdot)} \norm{y}.
		\]

		\item\label{enum1:Abb_linear_2Arg-Spezialfall-est_hohes_Diff}
		Suppose that there exist a normed space $\widetilde{X}$, a map $g \in \FC{\UF}{\widetilde{\SX}}{k}$
		and a continuous bilinear map $b : \widetilde{\SX} \times \SY \to \SZ$
		such that $\Xi = b \circ (g \times \id{\SY})$.
		Then
		\[
			\parFAbl[\ell]{1}{\Xi}(x, y)\eval(h_1, \dotsc, h_\ell)
			= b (\FAbl[\ell]{g}(x)\eval(h_1, \dotsc, h_\ell), y),
		\]
		for $h_1, \dotsc, h_\ell \in \SX$.
		In particular,
		\[
			\label{est:Abb_linear_2Arg-Spezialfall-hohes_Diff--partiell}
			\tag{\ensuremath{\dagger\dagger\dagger}}
			\Opnorm{\parFAbl[\ell]{1}{\Xi}(x, \cdot)}
			\leq \Opnorm{b} \Opnorm{\FAbl[\ell]{g}(x)}
		\]
		and (if $\ell \geq 1$)
		\begin{equation}\label{est:Abb_linear_2Arg-Spezialfall-hohes_Diff}
			\Opnorm{\FAbl[\ell]{\Xi}(x, y)}
			\leq \Opnorm{b} \ell \Opnorm{\FAbl[\ell - 1]{g}(x)}
			+ \Opnorm{b} \norm{y} \,\Opnorm{\FAbl[\ell]{g}(x)}.
		\end{equation}
	\end{assertions}
\end{lem}
\begin{proof}
	\ref{enum1:Erstes_par_Diff-Abb_linear_2Arg}
	We prove by induction on $\ell$ that $\dA[\ell]{_1\Xi}{}{}$ is linear in its second argument.
	For $\ell = 0$, this is true by our assumption.

	$\ell \to \ell + 1$:
	Since for $h_1, \dotsc, h_{\ell + 1} \in \SX$,
	\[
		\dA[\ell + 1]{_1\Xi}{x, y}{h_1, \dotsc, h_{\ell + 1}}
		= \AblCurve{t}\dA[\ell]{_1\Xi}{x + t h_{\ell + 1}, y }{h_1, \dotsc, h_\ell},
	\]
	and $\dA[\ell]{_1\Xi}{}{}$ is linear in its second argument,
	also $\dA[\ell + 1]{_1\Xi}{}{}$ is so.

	We prove \ref{id:Ableitung_Abb_linear_2Arg}.
	We get using the linearity of $\parFAbl[\ell]{1}{\Xi}$ in the second argument
	\[
		\AblCurve{t} \parFAbl[\ell]{1}{\Xi}(x + t h_1, y + t h_2)
		= \lim_{t \to 0} \parFAbl[\ell]{1}{\Xi}(x + t h_1, h_2) + \AblCurve{t} \parFAbl[\ell]{1}{\Xi}(x + t h_1, y)
	\]
	Since $\lim_{t \to 0} \parFAbl[\ell]{1}{\Xi}(x + t h_1, h_2) = \parFAbl[\ell]{1}{\Xi}(x, h_2)$
	and
	\[
		\AblCurve{t} \parFAbl[\ell]{1}{\Xi}(x + t h_1, y) \eval (v_1, \dotsc, v_\ell)
		= \parFAbl[\ell + 1]{1}{\Xi}(x, y)(v_1, \dotsc, v_\ell, h _1),
	\]
	for $v_1, \dotsc, v_\ell \in \SX$, the desired identity follows.

	\ref{enum1:Formel-est_hohe_Diff-Abb_linear_2Arg}
	We prove the identity for $\FAbl[\ell]{\Xi}$ by induction on $\ell$.

	$\ell = 1$:
	This follows directly from \ref{id:Ableitung_Abb_linear_2Arg}.

	$\ell \to \ell + 1$:
	We calculate the $(\ell + 1)$-th derivative of $\Xi$ using the inductive hypothesis
	and \ref{id:Ableitung_Abb_linear_2Arg}:
	\begin{align*}
		&\FAbl[\ell + 1]{\Xi}(x, y)\eval(h^1, \dotsc, h^{\ell + 1})
		\\
		=& \AblCurve{t} \FAbl[\ell]{\Xi}(x + t h^{\ell + 1}_1, y + t h^{\ell + 1}_2)\eval(h^1, \dotsc, h^\ell)
		\\
		=& \AblCurve{t} \parFAbl[\ell]{1}{\Xi}(x  + t h^{\ell + 1}_1, y + t h^{\ell + 1}_2)\eval(h^1_1,\dotsc, h^\ell_1)
					+ \sum_{j = 1}^\ell \AblCurve{t} \parFAbl[\ell - 1]{1}{\Xi}(x + t h^{\ell + 1}_1, h^j_2)\eval \widehat{h_1^j}
		\\
		=& \parFAbl[\ell]{1}{\Xi}(x,  h^{\ell + 1}_2)\eval(h^1_1,\dotsc, h^\ell_1)
			+ \parFAbl[\ell + 1]{1}{\Xi}(x, y)\eval(h^1_1,\dotsc, h^\ell_1, h^{\ell + 1}_1)
			+ \sum_{j = 1}^\ell  \parFAbl[\ell]{1}{\Xi}(x, h^j_2)\eval \widehat{h_1^j},
	\end{align*}
	from which we derive the assertion.

	The estimate \ref{est:norm_l-te_Ableitung-Abb_linear_2Arg} follows directly from this identity.

	\ref{enum1:Abb_linear_2Arg-Spezialfall-est_hohes_Diff}
	We first prove the identity by induction on $\ell$. The assertion obviously holds for $\ell = 0$.

	$\ell \to \ell + 1$:
	We use the inductive hypothesis to calculate
	\begin{multline*}
		\parFAbl[\ell + 1]{1}{\Xi}(x, y)\eval(h_1, \dotsc, h_{\ell + 1})
		= \AblCurve{t} \parFAbl[\ell]{1}{\Xi}(x + t h_{\ell + 1}, y)\eval(h_1, \dotsc, h_{\ell})
		\\
		= \AblCurve{t} b(\FAbl[\ell]{g}(x + t h_{\ell + 1})\eval(h_1, \dotsc, h_{\ell}), y )
		= b(\FAbl[\ell + 1]{g}(x)\eval(h_1, \dotsc, h_{\ell + 1}), y ),
	\end{multline*}
	so the assertion is established.

	The estimate \ref{est:Abb_linear_2Arg-Spezialfall-hohes_Diff--partiell} follows directly from this identity.
	Furthermore, we derive \ref{est:Abb_linear_2Arg-Spezialfall-hohes_Diff} from \ref{est:norm_l-te_Ableitung-Abb_linear_2Arg}
	and \ref{est:Abb_linear_2Arg-Spezialfall-hohes_Diff--partiell}.
\end{proof}

\begin{lem}\label{lem:Abschaetzung_hoheDiffs_Spezialfall-linArg}
	Let $E$, $F$, $\SX$, $\SY$ and $\SZ$ be normed spaces,
	$\UF \subseteq \SX$ and $\VF \subseteq \SY$ open nonempty sets,
	$b : \Lin{\SY}{\SZ} \times E \to F$ continuous bilinear with $\Opnorm{b} \leq 1$
	and $\Xi \in \FC{\UF \times \VF}{\SZ}{\infty}$.
		We define
	\[
		\Xi^{(2)}_b : \UF \times \VF \times E \to F
		: (x, y, e) \mapsto b(\parFAbl{2}{\Xi}(x, y), e).
	\]
	Then $\Xi^{(2)}_b(\UF \times \VF \times \sset{0}) = \sset{0}$,
	and for each $\ell \in \N^*$, we have
	\[
		\Opnorm{\FAbl[\ell]{\Xi^{(2)}_b}(x, y, e)}
			\leq \ell \Opnorm{\FAbl[\ell]{\Xi}(x, y)}
			+ \norm{e} \,\Opnorm{\FAbl[\ell + 1]{\Xi}(x, y)}.
	\]
	Moreover, for each $R > 0$,
	\begin{equation}\label{est:Differential-MaMu_hohes_Diff_1-l-Norm}
		\hn{\Xi^{(2)}_b}{1_{\UF \times \VF \times \Ball[E]{0}{R} }}{\ell}
				\leq \ell \hn{\Xi}{1_{\UF \times \VF}}{\ell}
					+ R \hn{\Xi}{1_{\UF \times \VF}}{\ell + 1}.
	\end{equation}
\end{lem}
\begin{proof}
	We get from \ref{est:Abb_linear_2Arg-Spezialfall-hohes_Diff} that
	\[
		\Opnorm{\FAbl[\ell]{\Xi^{(2)}_b}(x, y, e)}
		\leq \ell \Opnorm{\FAbl[\ell - 1]{(\parFAbl{2}{\Xi})}(x, y)}
			+ \norm{e} \,\Opnorm{\FAbl[\ell]{(\parFAbl{2}{\Xi})}(x, y)}.
	\]
	Since
	\[
		\Opnorm{\FAbl[\ell]{(\parFAbl{2}{\Xi})}(x, y)}
		\leq \Opnorm{\FAbl[\ell]{(\FAbl{\Xi})}(x, y)}
		= \Opnorm{\FAbl[\ell + 1]{\Xi}(x, y)}
	\]
	for all $\ell \in \N^*$, we obtain the first estimate.
	\ref{est:Differential-MaMu_hohes_Diff_1-l-Norm} follows.
\end{proof}
\subsection{The superposition operator}
We prove the above assertion about the superposition,
using notation from \refer{lem:Abschaetzung_hoheDiffs_Spezialfall-linArg}.
The hardest part of the proof will be the examination of the superposition with $\Xi^{(2)}_M$.
\begin{prop}\label{prop:SuperpostionCWZweiVars-id}
	Let $\SX$, $\SY$ and $\SZ$ be normed spaces, $\UF \subseteq \SX$ an open nonempty subset,
	$\VF \subseteq \SY$ an open neighborhood of $0$ that is star-shaped with center $0$,
	$\GewFunk \subseteq \cl{\R}^\UF$ with $1_\UF \in \GewFunk$ and $k \in \cl{\N}$.
	Further, let $\Xi \in \FC{\UF \times \VF}{\SZ}{\infty}$
	such that $\Xi(\UF \times \sset{0}) = \sset{0}$.
	\begin{assertions}
		\item\label{ass1:est_FWerte_SP2V}
		For maps $\gamma, \eta : \UF \to \VF$ such that the line segment
		$\set{t \gamma + (1-t) \eta}{t \in [0,1]} \sub \VF^\UF$ and $f \in \GewFunk$,
		the estimate
		\begin{equation}\label{est:f0-Norm_SPid-Differenz}
			\hn{\Xi \circ (\id{\UF}, \gamma) - \Xi \circ (\id{\UF}, \eta)}{f}{0}
			\leq
			\hn{\parFAbl{2}{\Xi} }{1_{\UF \times \VF}}{0}
			\hn{\gamma - \eta}{f}{0}
		\end{equation}
		holds. In particular, for $\eta = 0$ we get
		\begin{equation}\label{est:f0-Norm_SPid}
			\hn{\Xi \circ (\id{\UF}, \gamma)}{f}{0}
			\leq
			\hn{\parFAbl{2}{\Xi} }{1_{\UF \times \VF}}{0}
			\hn{\gamma}{f}{0}.
		\end{equation}

		\item\label{ass1:Differential_SP2V+est}
		Let $\gamma \in \FC{\UF}{\VF}{1}$. Then
		\[
						\FAbl{(\Xi\circ (\id{\UF}, \gamma))}
			= \parFAbl{1}{\Xi}\circ (\id{\UF}, \gamma) + \parFAbl{2}{\Xi}\circ (\id{\UF}, \gamma) \MaMu \FAbl{\gamma} .
		\]
		The map $\parFAbl{1}{\Xi}$ maps $\UF \times \sset{0}$ to $0$,
		and for $f \in \GewFunk$, we have
		\begin{equation}\label{est:f1-Norm_SPid}
			\hn{\Xi \circ (\id{\UF}, \gamma)}{f}{1}
			\leq
			\hn{\Xi }{1_{\UF \times \VF}}{2}
			\hn{\gamma}{f}{0}
			+
			\hn{\parFAbl{2}{\Xi} }{1_{\UF \times \VF} }{0}
			\hn{\gamma}{f}{1}.
		\end{equation}

		\item\label{ass1:SP2V-2conds_def_glatt}
		Suppose that
		\begin{equation}\label{cond:est_weights_SP}
			(\forall f \in \GewFunk, \ell \in \N^*)
			(\exists g \in \ExtWeights{\GewFunk})
			\, \hn{\Xi}{1_{\UF \times \VF}}{\ell} \abs{f} \leq \abs{g}.
		\end{equation}
		Then the map
		\[
			\Xi_{\ast}:
			\CcFo{\UF}{\VF}{k} \to \CcF{\UF}{\SZ}{k}
			: \gamma \mapsto \Xi \circ (\id{\UF}, \gamma)
		\]
		is defined and smooth with
		\begin{equation}
			\label{id:Differential_SuperposCWZweiVars-id}
			\dA{\Xi_{\ast}}{\gamma}{\gamma_{1}}
			= (\dA{_{2}\Xi}{}{})_{\ast}(\gamma, \gamma_{1}).
		\end{equation}
	\end{assertions}
\end{prop}
\begin{proof}
	\ref{ass1:est_FWerte_SP2V}
	For each $x \in \UF$, we calculate
	\begin{equation*}
		\Xi(x, \gamma(x)) - \Xi(x, \eta(x))
		= \Mint{ \dA{_{2}\Xi}{x, t \gamma(x) + (1-t) \eta(x)}{\gamma(x) - \eta(x)} }{t}.
	\end{equation*}
	Hence for each $f\in\GewFunk$, we have
	\[
		\abs{f(x)}\, \norm{\Xi(x, \gamma(x)) - \Xi(x, \eta(x))}
		\leq \hn{\parFAbl{2}{\Xi} }{1_{\UF \times \VF} }{0}
		\abs{f(x)}\, \norm{\gamma(x) - \eta(x)}.
	\]
	From this estimate, we conclude that \ref{est:f0-Norm_SPid-Differenz} holds.

	\ref{ass1:Differential_SP2V+est}
	The identity for $\FAbl{(\Xi \circ (\id{\UF}, \gamma))}$
	follows from the Chain Rule.
	For $x \in \UF$ and $h\in\SX$, we have
	\[
		\parFAbl{1}{\Xi}(x, 0)\eval h = \dA{_{1}\Xi}{x, 0}{h}
		= \lim_{t\to 0}\frac{\Xi(x + t h, 0) - \Xi(x, 0) }{t} = 0,
	\]
	whence $\parFAbl{1}{\Xi}(x, 0) = 0$.
	We then get the estimate by applying \ref{est:f0-Norm_SPid} to the first summand.

	\ref{ass1:SP2V-2conds_def_glatt}
	We first prove by induction on $k$ that $\Xi_*$ is defined and continuous.

	$k=0$:
	We see with \ref{est:f0-Norm_SPid} that $\Xi_*$ is defined since
	\[
		\hn{\Xi \circ (\id{\UF}, \gamma)}{f}{0}
		\leq \hn{\Xi}{1_{\UF \times \VF}}{1} \hn{\gamma}{f}{0}
		\leq \hn{\gamma}{g}{0}.
	\]
	With a similar argument, we see using \ref{est:f0-Norm_SPid-Differenz}
	that $\Xi_*$ continuous
	since each $\gamma \in \CcFo{\UF}{\VF}{0}$
	has a convex neighborhood in $\CcFo{\UF}{\VF}{0}$.

	$k \to k +1$:
	We use \refer{lem:topologische_Zerlegung_von_CFk}.
	So all that remains to show is that
	$\FAbl{(\Xi \circ (\id{\UF}, \gamma))} \in \CcF{\UF}{\Lin{\SX}{\SZ}}{k}$
	and
	$
		\gamma \mapsto \FAbl{(\Xi \circ (\id{\UF}, \gamma))}
	$
	is continuous. We proved in \ref{ass1:Differential_SP2V+est} that
	\[
		\FAbl{(\Xi \circ (\id{\UF}, \gamma))}
		= \parFAbl{1}{\Xi}\circ (\id{\UF}, \gamma)
		+ \Xi^{(2)}_M \circ (\id{\UF}, \gamma, \FAbl{\gamma}),
	\]
	see \refer{lem:Abschaetzung_hoheDiffs_Spezialfall-linArg}
	for the definition of $\Xi^{(2)}_M$
	(here, $M$ denotes the composition of linear operators).
	We also proved in \ref{ass1:Differential_SP2V+est} that
	$\parFAbl{1}{\Xi}(\UF \times\sset{0}) = \sset{0}$,
	and obviously
	$\hn{\parFAbl{1}{\Xi}}{1_{\UF \times \VF}}{\ell}
	\leq \hn{\Xi}{1_{\UF \times \VF}}{\ell + 1}$ for all $\ell \in \N$.
	Hence we can use the inductive hypothesis to see that
	\[
		\CcFo{\UF}{\VF}{k + 1} \to \CcF{\UF}{\Lin{\SX}{\SZ}}{k}
		: \gamma \mapsto \parFAbl{1}{\Xi}\circ (\id{\UF}, \gamma)
	\]
	is defined and continuous.
	We examine $\Xi^{(2)}_M$. To this end, let $R > 0$.
	We see using \ref{est:Differential-MaMu_hohes_Diff_1-l-Norm} that for $\ell \in \N^*$ and $f \in \GewFunk$,
	\[
		\hn{\Xi^{(2)}_M}{1_{\UF \times \VF \times \Ball[\Lin{X}{Y}]{0}{R} }}{\ell} \abs{f}
			\leq \ell \hn{\Xi}{1_{\UF \times \VF}}{\ell} \abs{f}
							+ R \hn{\Xi}{1_{\UF \times \VF}}{\ell + 1} \abs{f}
			\leq \ell \abs{g_\ell} + R \abs{g_{\ell + 1}}.
	\]
	Here, $g_\ell, g_{\ell + 1} \in \ExtWeights{\GewFunk}$ exist by our assumptions.
	Hence in both cases, we can apply the inductive hypothesis to $\Xi^{(2)}_M$ and get (using
	\refer{lem:gewichtete_Abb_Produktisomorphie-endl} implicitly)
	that the map
	\[
		\CcFo{\UF}{\VF}{k}
		\times \CcFo{\UF}{\Ball[\Lin{\SX}{\SY}]{0}{R} }{k}
		\to
		\CcF{\UF}{\Lin{\SX}{\SZ} }{k}
		: (\gamma, \Gamma) \mapsto \Xi^{(2)}_M \circ (\id{\UF}, \gamma, \Gamma)
	\]
	is defined and continuous.
	Hence for each $\gamma \in \CcFo{\UF}{\VF}{k + 1}$,
	the map
	\[
		\set{ \eta \in  \CcFo{\UF}{\VF}{k + 1}}{ \hn{ \eta }{1_{\UF}}{1} <  \hn{ \gamma }{1_{\UF}}{1} + 1}
		\to \CcF{\UF}{\Lin{\SX}{\SZ}}{k}
		: \eta \mapsto \Xi^{(2)}_M \circ (\id{\UF}, \eta, \FAbl{\eta})
	\]
	is defined and continuous.
	Since $1_\UF \in \GewFunk$, the domain of this map is a neighborhood of $\gamma$.
	This finishes the proof.

	We pass on to prove the smoothness of $\Xi_{\ast}$.
	To do this, we have to examine $\dA{_{2}\Xi}{}{}$.
	Obviously $\dA{_{2}\Xi}{}{} = \Xi^{(2)}_\eval$,
	where $\eval$ denotes the evaluation of linear operators.
	Hence we can use a similar argument as above when discussing $\Xi^{(2)}_M$
	to see that
	\[
		(\dA{_{2}\Xi}{}{})_{\ast}
		: \CcFo{\UF}{\VF}{k} \times \CcF{\UF}{\SY}{k} \to \CcF{\UF}{\SZ}{k}
		: (\gamma, \gamma_1) \mapsto \dA{_{2}\Xi}{}{} \circ (\id{\UF}, \gamma, \gamma_1)
	\]
	is defined and continuous.
	Now let $\gamma \in \CcFo{\UF}{\VF}{k}$ and $\gamma_{1} \in \CcF{\UF}{\SY}{k}$.
	Since $\CcFo{\UF}{\VF}{k}$ is open,
	there exists an $r > 0$ such that $\set{\gamma + s \gamma_{1} }{ s \in \Ball[\K]{0}{r}} \subseteq \CcFo{\UF}{\VF}{k}$.
	We calculate for $x \in \UF$ and $t \in \Ball[\K]{0}{r}\setminus\sset{0}$ (using \refer{lem:gewichtete_Abb_Produktisomorphie-endl} implicitly) that
	\begin{align*}
		\frac{\Xi_\ast(\gamma + t\gamma_{1})(x) - \Xi_\ast(\gamma)(x) }{t}
		&= \frac{\Xi(x, \gamma(x) + t\gamma_{1}(x)) - \Xi(x, \gamma(x))}{t}
		\\
		&= \Mint{ \dA{_{2}\Xi}{x, \gamma(x) + s t \gamma_{1}(x)}{\gamma_{1}(x)} }{s}
		\\
		&= \Mint{( \dA{_{2}\Xi}{}{})_{\ast}(\gamma + s t \gamma_1, \gamma_1)(x) }{s}.
	\end{align*}
	Hence we can apply \cite[La.~3.2.13]{MR2952176}
	to see that
	\[
		\frac{\Xi_\ast(\gamma + t\gamma_{1}) - \Xi_\ast(\gamma) }{t}
		= \Mint{( \dA{_{2}\Xi}{}{})_{\ast}(\gamma + s t \gamma_1, \gamma_1) }{s}.
	\]
	Using \refer{lem:Stetigkeit_parameterab_Int}, we derive
	that $\Xi_{\ast}$ is $\ConDiff{}{}{1}$ and \ref{id:Differential_SuperposCWZweiVars-id} holds.

	We see with \ref{est:Differential-MaMu_hohes_Diff_1-l-Norm}
	(again, using that $\dA{_{2}\Xi}{}{} = \Xi^{(2)}_\eval$)
	that \ref{cond:est_weights_SP} holds for $\dA{_{2}\Xi}{}{}$
	on $\UF \times \VF \times \Ball{0}{R}$ for each $R > 0$.
	Since $1_{\UF} \in \GewFunk$, we have that
	$\CcFo{\UF}{\VF \times \SY}{k} = \bigcup_{R > 0} \CcFo{\UF}{\VF \times \Ball{0}{R}}{k}$.
	So with an easy induction argument we conclude
	(using \refer{lem:gewichtete_Abb_Produktisomorphie-endl})
	from \ref{id:Differential_SuperposCWZweiVars-id}
	that $\Xi_{\ast}$ is $\ConDiff{}{}{\ell}$ for each $\ell \in \N$ and hence smooth.
\end{proof}

\section{Weighted restricted products}
\label{sec:weighted_restricted_products}
We are ready to discuss restricted products of weighted function spaces.
As suggested in the introduction, for the sake of clarity we first take a more general approach.
\subsection{Restricted products for locally convex spaces with uniformly parameterized seminorms}
\begin{defi}[Restricted products]
	Let $I$ and $J$ be nonempty sets, $\famiI{E_i}$ be a family of locally convex spaces
	such that for each $i \in I$, there exists
	a family $\fami{p_{i, j}}{j}{J}$ of seminorms on $E_i$ that defines its topology.
	For each $j \in J$, we define the quasinorm
	\[
		p_j : \prod_{i \in I} E_i \to [0, \infty]
		: \famiI{x_i} \mapsto \sup_{i \in I} p_{i, j}(x_i).
	\]
	With these, we define
	\[
		\Linf{\famiI{E_i}} \ndef
		\set{x \in \prod_{i \in I} E_i }{(\forall j \in J)\,p_j(x) < \infty}.
	\]
	We shall use the same symbol, $p_j$, for the restriction of $p_j$ to $\Linf{\famiI{E_i}}$.
	Endowed with the seminorms $\set{p_j}{j \in J}$, the latter is a locally convex space.
	Note that the topology on $\Linf{\famiI{E_i}}$ is finer than the ordinary product topology,
	and strictly finer if $\set{i \in I}{E_i \neq \sset{0}}$ is infinite.
\end{defi}
\subsubsection*{On Lipschitz continuous functions to a restricted product}
Since the topology of $\Linf{\famiI{E_i}}$ generally is finer than the product topology,
a map whose component maps are continuous is not necessarily continuous.
But we can give a sufficient criterion for Lipschitz continuity.
First, we give the following definition.
\begin{defi}
	Let $X, Y$ be locally convex spaces, $U \sub X$ open, $\phi : U \to Y$ and $p \in \normsOn{Y}$, $q \in \normsOn{X}$.
	Then we set
	\[
		\Lip{q}{p}{\phi}
		\ndef \inf \set{ L \in [0, \infty]}{
			(\forall x, y \in U)\,  \norm{\phi(x) - \phi(y)}_p \leq L \norm{x - y}_q }.
	\]
	If $\Lip{q}{p}{\phi} < \infty$,
	then $ \norm{\phi(x) - \phi(y)}_p \leq \Lip{q}{p}{\phi} \norm{x - y}_q$ for all $x, y \in U$.
\end{defi}
\begin{lem}\label{lem:L-Stetigkeit_Abb_in_LinfProd}
	Let $V$ be a nonempty subset of the locally convex space $X$.
	Let $A : V \to \Linf{\famiI{E_i}}$ be a map such that
	\[
		(\forall j \in J) (\exists p^j \in\normsOn{X}) \,
		\sup_{i \in I} \Lip{p^j}{p_{i, j}}{\pi_i \circ A} < \infty,
	\]
	where for $i \in I$, $\pi_i : \prod_{j \in I} E_j \to E_i$ denotes the canonical projection.
	Then $A$ is continuous. In fact, $\Lip{p^j}{p_j}{A}\leq \sup_{i \in I} \Lip{p^j}{p_{i, j}}{\pi_i \circ A} $ for each $j \in J$.
\end{lem}
\begin{proof}
	Let $x, y \in V$ and $j \in J$. We have
	\[
		\norm{A(x) - A(y)}_{p_j}
		= \sup_{i \in I} \norm{\pi_i (A(x)) - \pi_i(A(y))}_{p_{i, j}}
		\leq \sup_{i\in I} \Lip{p^j}{p_{i, j}}{\pi_i \circ A} \norm{x - y}_{p^j}.
	\]
	This finishes the proof.
\end{proof}
\subsubsection*{On the product of restricted products}
We turn to the product $\Linf[J_E]{\famiI{E_i}} \times \Linf[J_F]{\famiI{F_i}}$ of two restricted products.
If the seminorms of both spaces are indexed over the same set, it is isomorphic to another restricted product.
As a preparation, we make the following remark.
\begin{bem}
	For the following, note that if the locally convex spaces $E$ and $F$
	both have a generating family $\fami{p_j^E}{j}{J}$ and $\fami{p_j^F}{j}{J}$ of seminorms indexed over $J$,
	then there exists a generating family of seminorms for $E \times F$ that is indexed over $J$.
	For example, the family $\fami{\max \circ (p_j^E \times p_j^F)}{j}{J}$ generates the product topology on $E \times F$.
\end{bem}
\begin{lem}\label{lem:pktwProduktLInf}
	The sets
	$\Linf{\famiI{E_i \times F_i}}$
	and $\Linf{\famiI{E_i}} \times \Linf{\famiI{F_i}}$
	are isomorphic as topological vector spaces. The canonical isomorphism is the map
	\[
		\Linf{\famiI{E_i \times F_i}}
		\to \Linf{\famiI{E_i}} \times \Linf{\famiI{F_i}}
		:
		\famiI{e_i, f_i} \mapsto (\famiI{e_i}, \famiI{f_i})
		,
	\]
	and
	\[
		\Linf{\famiI{E_i}} \times \Linf{\famiI{F_i}} \to \Linf{\famiI{E_i \times F_i}}
		:
		(\famiI{e_i}, \famiI{f_i}) \mapsto \famiI{e_i, f_i}
	\]
	its inverse.
\end{lem}
\begin{proof}
	We denote the maps defined above by $A$ and $B$, respectively.
	Let $j \in J$ and $k \in I$.
	Then
	\[
		p^E_{k, j}((\pi_k \circ \mathrm{pr}_1 \circ A)\famiI{e_i, f_i})
		= p^E_{k, j}(e_k)
		\leq \max(p^E_{k, j}(e_k), p^F_{k, j}(f_k))
		\leq \max  (p^E_{j} \times p^F_{j})\famiI{e_i, f_i}
		,
	\]
	independent of $k$.
	This shows that $\mathrm{pr}_1 \circ A$ takes values in $\Linf{\famiI{E_i}}$, and since it is linear,
	we can use \refer{lem:L-Stetigkeit_Abb_in_LinfProd} to see that it is continuous to this space.
	Since the same argument can be made for the second factor, we see that $A$ is continuous.

	On the other hand, we have that
	\begin{multline*}
		\max\circ(p^E_{k, j} \times p^F_{k, j})((\pi_k \circ B)(\famiI{e_i}, \famiI{f_i}))
		= \max(p^E_{k, j}(e_k), p^F_{k, j}(f_k))
		\\
		\leq p^E_{k, j}(e_k) + p^F_{k, j}(f_k)
		\leq p^E_{j}\famiI{e_i} + p^F_{j}\famiI{f_i}
		.
	\end{multline*}
	Since $p^E_{j} \circ \mathrm{pr}_1 + p^F_{j} \circ \mathrm{pr}_2$
	is a continuous seminorm on $\Linf{\famiI{E_i}} \times \Linf{\famiI{F_i}}$,
	this shows that $B$ takes values in $\Linf{\famiI{E_i \times F_i}}$, and since it is linear,
	we can use \refer{lem:L-Stetigkeit_Abb_in_LinfProd} to see that it is continuous to this space.
	Now clearly $B = A^{-1}$.
\end{proof}
\subsubsection*{On differentiable functions into a restricted product}
We give a criterion when a function into a restricted product whose component maps are $\ConDiff{}{}{1}$
is differentiable itself. In order to do this, we give a sufficient condition for the completeness of a restricted product.
\paragraph{Completeness of a restricted product}
We prove that a restricted product is complete if all factors are so.
\begin{lem}[Completeness]\label{lem:Linf_compl_wenn_Faktoren_c}
	Let $I$ and $J$ be nonempty sets, $\famiI{E_i}$ be a family of locally convex spaces
	and $\fami{p_{i, j}}{j}{J}$ a family of generating seminorms for $E_i$, for $ i \in I$.
	Further assume that each $E_i$ is complete. Then $\Linf{\famiI{E_i}}$ is complete.
\end{lem}
\begin{proof}
	Let $\fami{x_\alpha}{\alpha}{A}$ be a Cauchy net in $\Linf{\famiI{E_i}}$.
	Then for each $i \in I$, obviously $\fami{\pi_i(x_\alpha)}{\alpha}{A}$
	is a Cauchy net in $E_i$, and since $E_i$ is complete, it converges to some $x_i \in E_i$.
	We show that $\famiI{x_i} \in \Linf{\famiI{E_i}}$ and that $\fami{x_\alpha}{\alpha}{A}$
	converges to $\famiI{x_i}$. To this end, let $j \in J$.
	Since $\fami{x_\alpha}{\alpha}{A}$ is a Cauchy net,
	for each $\eps > 0$ there exists $\ell \in A$ such that
	\[
		(\forall \alpha, \beta \in A : \alpha, \beta \geq \ell)
		\, \sup_{i \in I} \norm{\pi_i(x_\alpha) - \pi_i(x_\beta)}_{p_{i, j}} < \eps .
	\]
	We fix $\alpha$ in this estimate, and for each $i \in I$, we take $\pi_i(x_\beta)$ to its limit.
	Then we get that
	\[
		(\forall \alpha \in A : \alpha \geq \ell)
		\, \sup_{i \in I} \norm{\pi_i(x_\alpha) - x_i }_{p_{i, j}} \leq \eps .
	\]
	Hence
	\[
		\norm{ \famiI{x_i} }_{p_{j}}
		\leq \norm{ x_\ell }_{p_{j}} + \norm{ \famiI{x_i} - x_\ell }_{p_{j}}
		< \infty
	\]
	and thus $\famiI{x_i} \in \Linf{\famiI{E_i}}$.
	Since $\eps > 0$ was arbitrary,
	we also see that $\fami{x_\alpha}{\alpha}{A}$ converges to $\famiI{x_i}$.
\end{proof}
\paragraph{Differentiability criterion}
The criterion we present is quite useful.
The reason for this is that often,
we can compute the differentials in terms of the map itself and some well-behaved operations.
\begin{lem}\label{lem:Abb_nach_Linf_Ck_wenn_Komp_Ck_mit_stetigem_Diff}
	Let $U$ be an open nonempty subset of the locally convex space $E$,
	$I$ and $J$ nonempty sets, $\famiI{F_i}$ a family of locally convex spaces
	whose topologies are generated by families of seminorms indexed over $J$.
	Let $f : U \to \Linf{\famiI{F_i}}$ be a map
	such that each component map $f_i : U \to F_i$ is $\ConDiff{}{}{1}$ and the map
	\[
		\famiI{\dA{f_i}{}{}} : U \times E \to \Linf{\famiI{F_i}}
	\]
	is defined and continuous.
	Then $f$ is $\ConDiff{}{}{1}$.
\end{lem}
\begin{proof}
	Let $x \in U$ and $h \in E$. Choose $\eps > 0$ so small
	that $x + \Ball[\K]{0}{\eps} h \sub U$.
	By our assumptions, the map
	\[
		\Ball[\K]{0}{\eps} \times [0, 1] \to \Linf{\famiI{F_i}}
		: (t, s) \mapsto \famiI{\dA{f_i}{x + s t h}{h}}
	\]
	is continuous.
	Hence we see with \refer{lem:Linf_compl_wenn_Faktoren_c} that
	for each $t \in \Ball[\K]{0}{\eps}$,
	$\Mint{ \famiI{\dA{f_i}{x + s t h}{h}} }{s}$ exists in $\Linf{\famiI{\widetilde{F_i}}}$,
	where $\widetilde{F_i}$ denotes the completion of $F_i$.
	Using the mean value theorem, we conclude that the integral exists
	in $\Linf{\famiI{F_i}}$ with the value $\frac{1}{t}(f(x + t h) - f(x))$, if $t \neq 0$.
	Hence we see with the continuity of parameter-dependent integrals
	(\refer{lem:Stetigkeit_parameterab_Int})
	that $f$ is $\ConDiff{}{}{1}$ with
	$\dA{f}{x}{h} = \famiI{\dA{f_i}{x}{h} }$.
\end{proof}

\subsubsection*{On the product of multilinear maps}
The last result about the general restricted products is about the continuity of a product of multilinear maps.
It assures the continuity if the factors maps are kind of \enquote{uniformly bounded}
for each generating seminorm of the restricted product.
\begin{lem}[Multilinear maps]\label{lem:m-lin_Abb_glm_stetig->Prod_stetig}
	Let $I$ and $J$ be nonempty sets, $m \in \N$, $E_1, \dotsc, E_m$ be locally convex spaces
	and $\famiI{F_i}$ a family of locally convex spaces
	such that the topology of each $F_i$ is generated by a family $\fami{p_{i, j}}{j}{J}$ of seminorms.
	Further, for each $i \in I$ let $\beta_i : E_1 \times \dotsm \times E_m \to F_i$
	be an $m$-linear map such that
	\begin{equation*}
		\tag{\ensuremath{\dagger}}
		\label{cond:glm_Stetigkeit_m-lin_Abb}
		\begin{multlined}[0.92\columnwidth]
			(\forall j \in J) (\exists p_1 \in \normsOn{E_1}, \dotsc, p_m \in \normsOn{E_m},\, C > 0)
			\\
			(\forall i \in I, x_1 \in E_1, \dotsc, x_m \in E_m)\;
			\norm{\beta_i(x_1, \dotsc, x_m)}_{p_{i, j}} \leq C \norm{x_1}_{p_1} \dotsb \norm{x_m}_{p_m}.
		 \end{multlined}
	\end{equation*}
	Then the map
	\[
		\famiI{\beta_i} : E_1 \times \dotsc \times E_m \to \Linf{\famiI{F_i}}
	\]
	is defined, $m$-linear and continuous.
\end{lem}
\begin{proof}
	We conclude from \ref{cond:glm_Stetigkeit_m-lin_Abb} that for $j \in J$ and $x_1 \in E_1, \dotsc, x_m \in E_m$,
	\[
		\norm{\famiI{\beta_i(x_1, \dotsc, x_m)}}_{p_j}
		\leq C \norm{x_1}_{p_1} \dotsb \norm{x_m}_{p_m}.
	\]
	From this estimate, we conclude that $\famiI{\beta_i(x_1, \dotsc, x_m)} \in \Linf{\famiI{F_i}}$.
	Further, since $\famiI{\beta_i}$ is obviously $m$-linear,
	we see that it is continuous in $0$ and hence continuous.
\end{proof}

\subsection{Restricted products of weighted functions}
We now turn our attention to special restricted products,
where each factor is a weighted function space of the kind examined in \cite[Chpt.~3]{MR2952176}.
Since we know the topology of these spaces and plenty of operations on and between them very well,
we are able to derive more results about them than in the general case.
We give the definition and then adapt some previous results about the topological and uniform structure.
\subsubsection{Definition, topological and uniform structure}
\begin{defi}
	Let $I$ be a nonempty set, $\famiI{U_i}$ a family such that each $U_i$
	is an open nonempty set of a normed space $X_i$,
	$\famiI{Y_i}$ another family of normed spaces,
	$\GewFunk \sub \cl{\R}^{\disjointUI{U_i}}$ a nonempty family of weights
	defined on the disjoint union $\disjointUI{U_i}$ of $\famiI{U_i}$,
	and $k \in \cl{\N}$.
	For $i \in I$ and $f \in \GewFunk$, we set $f_i \ndef \rest{f}{U_i}$,
	and further $\GewFunk_i \ndef \set{f_i}{f \in \GewFunk}$.
	Then the topology of each space $\CF{U_i}{Y_i}{\GewFunk_i}{k}$
	is induced by a family of seminorms
	indexed over $\GewFunk \times \set{\ell \in \N}{\ell \leq k}$;
	for $i \in I$, we map $f \in \GewFunk$ and $\ell \in \N$ with $\ell \leq k$ to $\hn{\cdot}{f_i}{\ell}$.
	We define
	\[
		\CcFproI{U_i}{Y_i}{k}
				\ndef
		\Linf[\set{\hn{\cdot}{f}{\ell} }{(f, \ell) \in \GewFunk \times \set{n \in \N}{\,n \leq k}}]{\famiI{ \CF{U_i}{Y_i}{\GewFunk_i}{k} }}.
	\]
	The seminorms that generate the topology on this space are of the form
	\[
		\hn{\famiI{\phi_i} }{f}{\ell}
		\ndef
		\sup_{i \in I} \hn{\phi_i}{f_i}{\ell},
	\]
	where $f \in \GewFunk$ and $\ell \in \N$ with $\ell \leq k$.
\end{defi}

\begin{lem}\label{lem:CinfLinf_initial_CkLinf}
	$\CcFproI{U_i}{Y_i}{\infty}$ is endowed with the initial topology
	of the inclusion maps
	\[
		\CcFproI{U_i}{Y_i}{\infty}
		\to
		\CcFproI{U_i}{Y_i}{k},
	\]
	for $k \in \N$. Moreover,
	$
		\CcFproI{U_i}{Y_i}{\infty} = \varprojlim_{k \in \N} \CcFproI{U_i}{Y_i}{k}
	$.
\end{lem}
\begin{proof}
	This is clear from the fact that the seminorms
	$\hn{\cdot}{f}{\ell}$ with $f \in \GewFunk$ and $\ell \leq k$
	define the topology on the right hand side,
	while those with $\ell \in \N$ define the topology on the left.
\end{proof}

\begin{prop}\label{prop:Zerlegungssatz_Familie}
	Let $k \in \N$. Then for $\famiI{\phi_i} \in \prod_{i \in I} \FC{U_i}{Y_i}{1}$,
	we have
	\[
		\famiI{\phi_i} \in \CcFproI{U_i}{Y_i}{k + 1}
		\iff
		\famiI{\phi_i} \in \CcFproI{U_i}{Y_i}{0}
		\text{ and }
		\famiI{\FAbl{\phi_i}} \in \CcFproI{U_i }{ \Lin{X_i}{Y_i} }{k}.
	\]
	The map
	\begin{equation*}
		\CcFproI{U_i}{Y_i}{k + 1}
		\to
		\CcFproI{U_i}{Y_i}{0}
		\times
		\CcFproI{U_i }{ \Lin{X_i}{Y_i} }{k}
		:
		(\famiI{\phi_i}) \mapsto (\famiI{\phi_i}, \famiI{\FAbl{\phi_i} })
	\end{equation*}
	is linear and a topological embedding.
\end{prop}
\begin{proof}
	This is proved in the same way as \refer{lem:topologische_Zerlegung_von_CFk}.
\end{proof}

\paragraph{Lipschitz continuity}
This is an adaptation of \refer{lem:L-Stetigkeit_Abb_in_LinfProd}.
\begin{lem} \label{lem:L-Stetigkeit_Abb_in_LinfProd-gewAbb}
	Let $V$ be an open nonempty subset of the locally convex space $X$.
	Let $A : V \to \CcFproI{U_i}{Y_i}{k}$ be a map such that
	\[
		(\forall f \in \GewFunk, \ell \in \N : \ell \leq k) (\exists p \in\normsOn{X}) \,
		\sup_{i \in I} \Lip{p}{f_i, \ell}{\pi_i \circ A} <\infty.
	\]
	Then $A$ is continuous. In fact, $\Lip{p}{f, \ell}{A}\leq \sup_{i \in I} \Lip{p}{f_i, \ell}{\pi_i \circ A} $.
\end{lem}
\begin{proof}
	This follows from \refer{lem:L-Stetigkeit_Abb_in_LinfProd}.
\end{proof}
\subsubsection{Adjusting weights and open subsets}
Let $I$ be an infinite set and $\famiI{r_i}$ a family of positive real numbers
such that $\inf_{i \in I} r_i = 0$.
If $\GewFunk$ consists only of $1_{\disjointUI{U_i}}$, then the set
$\prod_{i \in I} \CF{U_i}{\Ball[Y_i]{0}{r_i}}{\GewFunk_i}{0}$
is not a neighborhood of $0$ in $\CcFproI{U_i}{Y_i}{0}$.
But since we later need to discuss such sets,
and in particular want functions that are defined on such sets to be differentiable
(think of the Riemannian exponential function),
we must know under which conditions on $\GewFunk$ their interior is not empty.

It turns out that if $\GewFunk$ contains a weight $\omega$ that is \enquote{large enough} on each $U_i$,
then the set $\set{\famiI{\phi_i} \in \CcFproI{U_i}{Y_i}{0} }{\hn{\famiI{\phi_i}}{\omega}{0} < 1}$ is contained in
$\prod_{i \in I} \CF{U_i}{\Ball[Y_i]{0}{r_i}}{\GewFunk_i}{0} \cap \CcFproI{U_i}{Y_i}{0}$,
so the latter is a neighborhood of $0$.
We will call $\omega$ adjusting to the family $\famiI{r_i}$ since $\omega$ adjusts its smallness.
We start with some definitions.
\begin{defi}
	Let $\famiI{U_i}$ and $\famiI{r_i}$ be families such that each $U_i$
	is an open nonempty set of the normed space $X_i$, and each $r_i \in ]0, \infty]$.
	We say that $\omega : \disjointUI{U_i} \to \R$ is
	an \emph{adjusting weight for $\famiI{r_i}$} if for each $i \in I$,
	we have that
	\[
		\sup_{x \in U_i} \abs{\omega_i(x)} < \infty
		\qquad\text{and}\qquad
		\inf_{x \in U_i} \abs{\omega_i(x)} \geq \max\bigl(\tfrac{1}{r_i}, 1\bigr).
	\]
	Notice that generally, $\omega$ itself is \emph{not} bounded.
\end{defi}

\begin{defi}
	Let $\famiI{U_i}$ and $\famiI{V_i}$ be families such that each $U_i$
	is an open nonempty set of the normed space $X_i$
	and each $V_i$ is an open nonempty subset of a normed space $Y_i$,
	$\GewFunk \sub \cl{\R}^{\disjointUI{U_i}}$ a nonempty set
	and $k \in \cl{\N}$.
	Let $\omega : \disjointUI{U_i} \to \R$ with $0 \notin \omega(\disjointUI{U_i})$.
	We set
	\begin{multline*}
		\CcFfoPROi{U_i}{V_i}{k}{\omega}
		\\
		\ndef
		\set{\famiI{\gamma_i} \in \CcFproI{U_i}{Y_i}{k}}%
		{(\exists r > 0)(\forall i \in I, x \in U_i)\,\gamma_i(x) + \Ball[Y_i]{0}{\tfrac{r}{\abs{\omega(x)}}} \sub V_i}.
	\end{multline*}
	In particular, we define
	\[
		\CcFoPROi{U_i}{V_i}{k} \ndef \CcFfoPROi{U_i}{V_i}{k}{(1_{\disjointUI{U_i}})}.
	\]
	Additionally, if each $V_i$ is star-shaped with center $0$,
	then $\omega$ is called an \emph{adjusting weight for $\famiI{V_i}$} if it is an adjusting weight for
	$
		\famiI{\dist{\sset{0}}{\partial V_i}}.
	$
	If it is clear to which family $\omega$ adjusts, we may call $\omega$ just an adjusting weight.
\end{defi}

\begin{bem}\label{bem:CcFfoPRO_pktw_in_CcFo}
	Let $\famiI{U_i}$ and $\famiI{V_i}$ be families such that all $U_i$ and $V_i$
	are open nonempty subsets of the normed spaces $X_i$ respectively $Y_i$,
	$\GewFunk \sub \cl{\R}^{\disjointUI{U_i}}$ a nonempty set, $k \in \cl{\N}$
	and $\omega : \disjointUI{U_i} \to \R$ with $0 \notin \omega(\disjointUI{U_i})$
	such that $\sup_{x \in U_i}\abs{\omega_i(x)} < \infty$ for each $i \in I$.
	Then $\inf_{x \in U_i} \tfrac{1}{\abs{\omega_i}(x)} > 0$, and hence
	\[
		\CcFfoPROi{U_i}{V_i}{k}{\omega}
		\sub \prod_{i \in I} \CcFo{ U_i }{ V_i }{k}.
	\]
\end{bem}
To show that $\prod_{i \in I} \CF{U_i}{\Ball[Y_i]{0}{r_i}}{\GewFunk_i}{0}$
contains a neighborhood of the constant $0$ function,
we estimate the $\hn{\cdot}{1_U}{0}$ seminorm with the $\hn{\cdot}{f}{0}$ seminorm.
\begin{lem}\label{lem:est_1-0-norm_f-0-norm}
	Let $X$ and $Y$ be normed spaces, $U \sub X$ an open nonempty set,
	$f  : U \to \cl{\R}$ such that $0 \notin f(U)$ and $\phi, \psi : U \to Y$.
	\begin{assertions}
		\item\label{ass1:est_1-0-norm_f-0-norm_pktw}
		For all $x \in U$, we have $\norm{\phi(x) - \psi(x)} \leq \frac{\hn{\phi - \psi}{f}{0}}{\abs{f(x)}}$.

		\item\label{ass1:est_1-0-norm_f-0-norm}
		Assume that $\inf_{x \in U} \abs{f(x)} > 0$. Then
		$
			\hn{\phi - \psi}{1_U}{0} \leq \frac{\hn{\phi - \psi}{f}{0}}{ \inf_{x \in U} \abs{f(x)} }.
		$

		\item\label{ass1:est_1-0-norm_f-0-norm_spezielles-f}
		Suppose that $\inf_{x \in U} \abs{f(x)} \geq \max(\tfrac{1}{d}, 1)$,
		where $d > 0$. Then
		\begin{equation}\label{est:1-0-norm_f-0-norm_spezielles-f}
			\hn{\phi - \psi}{1_U}{0} \leq \min(d, 1) \hn{\phi - \psi}{f}{0}.
		\end{equation}
	\end{assertions}
\end{lem}
\begin{proof}
	\ref{ass1:est_1-0-norm_f-0-norm_pktw}
	This follows from $\abs{f(x)} \,\norm{\phi(x) - \psi(x)} \leq \hn{\phi - \psi}{f}{0}$.

	\ref{ass1:est_1-0-norm_f-0-norm}
	This is an easy consequence of \ref{ass1:est_1-0-norm_f-0-norm_pktw}.

	\ref{ass1:est_1-0-norm_f-0-norm_spezielles-f}
	This follows from \ref{ass1:est_1-0-norm_f-0-norm}, where we use that
	$\tfrac{1}{\max(\tfrac{1}{d}, 1) } = \min(d, 1)$.
\end{proof}

\begin{lem}
	Let $\famiI{U_i}$ and $\famiI{V_i}$ be families such that each $U_i$
	is an open nonempty set of a normed space $X_i$
	and each $V_i$ is an open nonempty subset of a normed space $Y_i$,
	$k \in \cl{\N}$, $f : \disjointUI{U_i} \to \R$ with $0 \notin f(\disjointUI{U_i})$
	and $\GewFunk \sub \cl{\R}^{\disjointUI{U_i}}$ with $f \in \GewFunk$.
	\begin{assertions}
		\item\label{ass1:CFof_offen}
		$\CcFfoPROi{U_i}{V_i}{k}{f}$ is open in $\CcFproI{U_i}{Y_i}{k}$.
		In fact, it is even open in  $\CcFproI{U_i}{Y_i}{k}$ when this space
		is endowed with the topology of $\CFproI{U_i}{Y_i}{\sset{f}}{0}$.

		\item\label{ass1:CFof_offen_nichtleer_f-gut}
		Assume that each $V_i$ is star-shaped with center $0$ and $f$ is an adjusting weight for $\famiI{V_i}$.
		Then $\CcFfoPROi{U_i}{V_i}{k}{f}$ is not empty. In particular, for $\tau > 0$ we have
		\begin{equation}\label{incl:1-Kugel_f0-norm_sub_CFof}
			\set{\eta \in \CcFproI{U_i}{Y_i}{k}}{ \hn{\eta}{f}{0} < \tau}
			\sub \CcFfoPROi{U_i}{\tau \cdot V_i}{k}{f}.
		\end{equation}
	\end{assertions}
\end{lem}
\begin{proof}
	\ref{ass1:CFof_offen}
	Let $\gamma \in \CcFfoPROi{U_i}{V_i}{k}{f}$. Then there exists $r > 0$ such that
	\[
		(\forall i \in I,  x \in U_i)\, \gamma_i(x) + \Ball[Y_i]{0}{\tfrac{r}{\abs{f(x)}}} \sub V_i.
	\]
	We show that
	\[
		\set{\eta \in \CcFproI{U_i}{Y_i}{k}}{ \hn{\eta - \gamma}{f}{0} < r}
		\sub \CcFfoPROi{U_i}{V_i}{k}{f}.
	\]
	To this end, let $\eta$ be an element of set on the left hand side and $s \ndef r - \hn{\eta - \gamma}{f}{0}$.
	Then for $i \in I$, $x \in U_i$ and $h \in \Ball[Y_i]{0}{\tfrac{s}{\abs{f(x)}}}$,
	we have with \refer{lem:est_1-0-norm_f-0-norm} and the triangle inequality
	\[
		\norm{\eta_i(x) - \gamma_i(x) + h}
		< \frac{\hn{\gamma - \eta}{f}{0} }{\abs{f(x)}} + \frac{s}{\abs{f(x)}}
		= \frac{r}{\abs{f(x)}}.
	\]
	Hence
	\[
		\eta_i(x) + h = \gamma_i(x) + \eta_i(x) - \gamma_i(x) + h \in V_i.
	\]
	This shows that $\eta \in \CcFfoPROi{U_i}{V_i}{k}{f}$.

	\ref{ass1:CFof_offen_nichtleer_f-gut}
	Let $\eta$ be an element of the set on the left hand side of \ref{incl:1-Kugel_f0-norm_sub_CFof}.
	We set $r \ndef \tau - \hn{\eta}{f}{0}$. Let $i \in I$, $x \in U_i$ and $h \in \Ball[Y_i]{0}{\tfrac{r}{\abs{f(x)}}}$.
	Then we see with \ref{est:1-0-norm_f-0-norm_spezielles-f} that
	\[
		\norm{\eta_i(x) + h}
		\leq \norm{\eta_i(x)} + \norm{h}
		< \min(1, d_i) \hn{\eta}{f}{0} + \min(1, d_i) (\tau - \hn{\eta}{f}{0}),
	\]
	where $d_i \ndef \dist{\sset{0}}{\partial V_i}$. Hence $\norm{\eta_i(x) + h} < \tau d_i$,
	so $\eta_i(x) + h \in \tau \cdot V_i$.
	This finishes the proof.
\end{proof}

\begin{bem}\label{bem:konstantes-1-Gew_adjust-weight}
	Let $\famiI{U_i}$ be a family such that each $U_i$
	is an open nonempty set of the normed space $X_i$.
	Further, let $\GewFunk \sub \cl{\R}^{\disjointUI{U_i}}$ contain $\omega$ with $\inf_{x \in U} \abs{\omega(x)} > 0$
	(in particular, this holds if $\omega$ is an adjusting weight) and $k \in \cl{\N}$.
	Then for each $\ell \in \N$ with $\ell \leq k$, we see with \refer{lem:est_1-0-norm_f-0-norm} that
	the seminorm $ \hn{\cdot}{1_{\disjointUI{U_i}} }{\ell} $ is continuous on $\CcFproI{U_i}{Y_i}{k}$.
	In particular, $\CcFproI{U_i}{Y_i}{k} = \CFproI{U_i}{Y_i}{\GewFunk \cup \sset{1_{\disjointUI{U_i}}}}{k}$.
\end{bem}
\subsection{Simultaneous superposition and multiplication}
In this subsection, we discuss operations between restricted products of weighted functions
that consist of operations that are defined on a single factor.
The most common operation is the superposition with a family $\famiI{\phi_i}$ of maps
of certain characteristics, i.e. linear, analytic etc.
In contrast to results derived in \cite{MR2952176},
we often have to take a more quantitative approach,
and tailor our assumptions about the permitted weights to $\famiI{\phi_i}$.
\subsubsection{Simultaneous multiplication}
We begin with simultaneous multiplication.
It is pretty straightforward, and \ref{cond:est_sim-multiplier_weights}
provides a good example of the assumptions on the weights that will be made in the following.
\begin{lem}\label{lem:simultane_mult-multiplier}
	Let $\famiI{U_i}$ be a family such that each $U_i$
	is an open nonempty set of the normed space $X_i$,
	and $\famiI{Y_i^1}$, $\famiI{Y_i^2}$, $\famiI{Z_i}$ be families of normed spaces.
	Further, for each $i \in I$ let $M_i : U_i \to Y_i^1$ be smooth,
	and $\beta_i : Y_i^1 \times Y_i^2 \to Z_i$ a bilinear map such that
	\[
		\sup\set{\Opnorm{\beta_i}}{i \in I} < \infty .
	\]
	Assume that $\GewFunk \sub \cl{\R}^{\disjointUI{U_i}}$ is nonempty and
	\begin{equation}\label{cond:est_sim-multiplier_weights}
			(\forall f \in \GewFunk, \ell \in \N)
			(\exists g \in \ExtWeights{\GewFunk})
			\, (\forall i \in I)\, \hn{M_i}{1_{U_i}}{\ell} \abs{f_i} \leq \abs{g_i}.
	\end{equation}
	Then for $k \in \cl{\N}$, the map
	\[
				\CcFproI{U_i}{Y_i^2}{k}
		\to
		\CcFproI{U_i}{Z_i}{k}
		:
		\famiI{\gamma_i} \mapsto \famiI{\beta_i \circ (M_i, \gamma_i)}
	\]
	is defined and continuous linear.
\end{lem}
\begin{proof}
	We prove this by induction on $k$.

	$k = 0$:
	We calculate for $i \in I$, $x \in U_i$, $\famiI{\gamma_i} \in \CcFproI{U_i}{Y_i^2}{k}$
	and $f \in \GewFunk$ that
	\[
		\abs{f_i(x)}\,\norm{(\beta_i \circ (M_i, \gamma_i))(x)}
		\leq
		\Opnorm{\beta_i} \, \abs{f_i(x)}\, \norm{M_i(x)}\, \norm{\gamma_i(x)}
		\leq
		\Opnorm{\beta_i} \, \hn{\gamma_i}{g_i}{0}.
	\]
	Hence
	\[
		\hn{\famiI{\beta_i \circ (M_i, \gamma_i)}}{f}{0}
		\leq \sup_{i\in I} \Opnorm{\beta_i} \, \hn{\famiI{\gamma_i}}{g}{0},
	\]
	which shows the assertion.

	$k \to k + 1$:
	Using the induction base and \refer{prop:Zerlegungssatz_Familie},
	all we have to show is that for $\famiI{\gamma_i} \in \CcFproI{U_i}{Y_i^2}{k}$, we have
	$\famiI{\FAbl{(b_i \circ (M_i, \gamma_i))} } \in \CcFproI{U_i}{\Lin{\SX_i}{\SZ_i}}{k}$
	and that the map
	\[
		\CcFproI{U_i}{Y_i^2}{k + 1} \to \CcFproI{U_i}{\Lin{\SX_i}{\SZ_i}}{k}
		:
		\famiI{\gamma_i} \mapsto \famiI{\FAbl{(b_i \circ (M_i, \gamma_i))} }
	\]
	is continuous. By \cite[La 3.3.2]{MR2952176},
	for each $i \in I$ we have
	\[
		\FAbl{(\beta_i \circ (M_i, \gamma_i))}
		= \beta_i^{(1)} \circ (\FAbl{M_i}, \gamma) +  \beta_i^{(2)} \circ (M_i, \FAbl{\gamma_i})
	\]
	(using notation as in \cite[Def 3.3.1]{MR2952176}). Hence
	\[
		\famiI{\FAbl{(\beta_i \circ (M_i, \gamma_i))}}
		= \famiI{\beta_i^{(1)} \circ (\FAbl{M_i}, \gamma)} +  \famiI{\beta_i^{(2)} \circ (M_i, \FAbl{\gamma_i})},
	\]
	and we easily calculate that $\Opnorm{\beta_i^{(1)}}, \Opnorm{\beta_i^{(2)}} \leq \Opnorm{\beta_i}$ for each $i \in I$.
	Since $\GewFunk$ and $\famiI{\FAbl{M_i}}$ satisfy \ref{cond:est_sim-multiplier_weights},
	we can apply the inductive hypothesis to both summands and finish the proof.
\end{proof}
\subsubsection{Simultaneous superposition with multilinear maps}
Here, we examine the superpositions with multilinear maps that are uniformly bounded.
It is very similar to \cite[Prop 3.3.3]{MR2952176},
but also involves a result for the more general restricted products defined above.
\begin{lem}\label{lem:multilineareSuperpos-Linf}
	Let $I$ be a nonempty set, $\famiI{X_i}$, $\fami{X_{i, k}}{(i, k)}{I \times \sset{1, \dotsc, n}}$ and $\famiI{Y_i}$
	families of normed spaces, and $U_i \sub X_i$ an open nonempty subset for each $i \in I$.
	Let $\cW_1, \dotsc, \cW_n, \GewFunk \sub \cl{\R}^{\disjointUI{U_i}}$
	be nonempty sets such that
	\[
		(\forall f \in \GewFunk) (\exists g^{f,1} \in \cW_1,\dotsc, g^{f,n} \in\cW_n)
		(\forall i \in I)\, \abs{f_i} \leq \abs{g^{f,1}_i}\dotsm \abs{g^{f, n}_i}.
	\]
	Further, for each $i \in I$, let $\beta_i : X_{i, 1} \times \dotsm \times X_{i, n} \to Y_i$
	be a continuous $n$-linear map such that the set
	\[
		\set{\Opnorm{\beta_i}}{i \in I}
	\]
	is bounded. Then the map
	\begin{align*}
		\beta : \CFproI{U_i }{ X_{i, 1} }{\cW_1}{k}
		\times \dotsm \times\
		\CFproI{U_i }{ X_{i, n} }{\cW_n}{k}
		&\to
		\CcFproI{U_i}{Y_i}{k}\\
		\famiI{\gamma_{i, 1}, \dotsc, \gamma_{i, n}}
		&\mapsto
		\famiI{\beta_i \circ (\gamma_{i, 1}, \dotsc, \gamma_{i, n})}
	\end{align*}
	is defined, $n$-linear and continuous.
\end{lem}
\begin{proof}
	Using \cite[Prop 3.3.3]{MR2952176},
	we have for each $i \in I$ and
	$\gamma_{i, 1} \in \CcF{U_i}{ X_{i, 1}}{k}$, \ldots, $\gamma_{i, n} \in  \CcF{U_i}{X_{i, n} }{k}$
	that
	$\beta_i \circ (\gamma_{i, 1}, \dotsc, \gamma_{i, n}) \in \CcF{U_i}{ Y_i }{k}$.
	Further, $\beta$ is $n$-linear as map to $\prod_{i \in I} \CcF{U_i}{ Y_i }{k}$.
	We prove by induction on $k$ that
	$\beta$ takes values in $\CcFproI{U_i}{Y_i}{k}$ and is continuous.

	$k = 0$:
	We compute for all $i \in I$, $f \in \GewFunk_i$ and
	$\gamma_{i, 1} \in \CF{U_i}{ X_{i, 1}}{\cW_1}{k}$, \ldots, $\gamma_{i, n} \in  \CF{U_i}{X_{i, n} }{\cW_n}{k}$
	that
	\[
		\hn{\beta_i \circ (\gamma_{i, 1}, \dotsc, \gamma_{i, n}) }{f}{0}
		\leq \Opnorm{\beta_i} \prod_{j = 1}^n \hn{\gamma_{i, j}}{ g^{f,j}_i }{0}.
	\]
	Since $i$ was arbitrary, we can apply \refer{lem:m-lin_Abb_glm_stetig->Prod_stetig} to derive the assertion.

	$k \to k + 1$:
	Using the induction base and \refer{prop:Zerlegungssatz_Familie},
	all we have to show is that for
	$\famiI{\gamma_{i, 1}} \in \CFproI{U_i}{ X_{i, 1} }{\cW_1}{k + 1}$, \ldots, $\famiI{\gamma_{i, n}} \in \CFproI{U_i}{ X_{i, n} }{\cW_n}{k + 1}$,
	\[
		\famiI{\FAbl{(\beta_i \circ (\gamma_{i, 1}, \dotsc, \gamma_{i, n}))}}
		\in \CcFproI{U_i}{ \Lin{X_i}{Y_i} }{k},
	\]
	and that the map
	\begin{align*}
		\CFproI{U_i }{ X_{i, 1} }{\cW_1}{k + 1}
		\times \dotsm \times
		\CFproI{U_i }{ X_{i, n} }{\cW_n}{k + 1}
		&\to
		\CcFproI{U_i }{ \Lin{X_i}{Y_i} }{k}\\
		\famiI{\gamma_{i, 1}, \dotsc, \gamma_{i, n}}
		&\mapsto
		\famiI{\FAbl{(\beta_i \circ (\gamma_{i, 1}, \dotsc, \gamma_{i, n}))}}
	\end{align*}
	is continuous. By \cite[La 3.3.2]{MR2952176},
	for each $i \in I$ we have
	\[
		\FAbl{(\beta_i \circ (\gamma_{i, 1}, \dotsc, \gamma_{i, n}))}
		= \sum_{j = 1}^n \beta_i^{(j)} \circ (\gamma_{i, 1}, \dotsc, \FAbl{\gamma_{i, j}},\dotsc, \gamma_{i, n})
	\]
	(using notation as in \cite[Def 3.3.1]{MR2952176}) and hence
	\[
		\famiI{\FAbl{(\beta_i \circ (\gamma_{i, 1}, \dotsc, \gamma_{i, n})})}
		= \sum_{j = 1}^n \famiI{\beta_i^{(j)} \circ (\gamma_{i, 1}, \dotsc, \FAbl{\gamma_{i, j}},\dotsc, \gamma_{i, n})}.
	\]
	Since we easily calculate that $\Opnorm{\beta_i^{(j)}} \leq \Opnorm{\beta_i}$ for each $i \in I$
	and $j \in \sset{1, \dotsc, n}$, we can apply the inductive hypothesis to each summand
	and get the assertion.
\end{proof}
\subsubsection{Simultaneous superposition with differentiable maps}
\label{sususec:Simultaneous_SuPo_smooth_maps}
We provide the simultaneous analogue of \refer{prop:SuperpostionCWZweiVars-id}.
In the proof, we have to use notation introduced in \refer{lem:Abschaetzung_hoheDiffs_Spezialfall-linArg},
as we did in the proof of \ref{prop:SuperpostionCWZweiVars-id}.
Similarly, the technically most challenging part will be
the examination of the superposition with $\famiI{(\beta_i)^{(2)}_{M_i}}$.
Another novelty is the use of adjusting weights.
\begin{prop}\label{prop:simultane_SP_BCinf0_Produkt}
	Let $\famiI{U_i}$ and $\famiI{V_i}$ be families such that each $U_i$
	is an open nonempty set of the normed space $X_i$
	and each $V_i$ is an open, star-shaped subset with center $0$ of a normed space $Y_i$.
	Further, let $\famiI{Z_i}$ be another family of normed spaces
	and $\GewFunk \sub \cl{\R}^{\disjointUI{U_i}}$ contain an adjusting weight $\omega$.
	For each $i \in I$, let $\beta_i \in \FC{ U_i \times V_i}{ Z_i}{\infty}$ be a map
	such that $\beta_i(U_i \times \sset{0}) = \sset{0}$.
	Further, assume that
	\begin{equation}\label{cond:est_SP-Abb_weights}
			(\forall f \in \GewFunk, \ell \in \N^*)
			(\exists g \in \ExtWeights{\GewFunk})
			\, (\forall i \in I)\, \hn{\beta_i}{1_{U_i \times V_i}}{\ell} \abs{f_i} \leq \abs{g_i}
	\end{equation}
	is satisfied.
	Then for $k \in \cl{\N}$, the map
	\[
		\beta_* \ndef \prod_{i \in I} (\beta_i)_* : \CcFfoPROi{U_i}{V_i}{k}{\omega}
		\to
		\CcFproI{U_i}{Z_i}{k}
		:
		\famiI{\gamma_i} \mapsto \famiI{\beta_i \circ (\id{U_i}, \gamma_i)}
	\]
	is defined and smooth.
\end{prop}
\begin{proof}
	We see with \refer{prop:SuperpostionCWZweiVars-id} (and \refer{bem:CcFfoPRO_pktw_in_CcFo})
	that $\beta_*$ is defined as a map to $\prod_{i \in I} \CcF{U_i}{Z_i}{k}$.
	We first prove by induction on $k$ that $\beta_*$ takes its values in $\CcFproI{U_i}{Z_i}{k}$
	and is continuous.

	$k = 0$:
	Let $f \in \GewFunk$.
	Using \ref{est:f0-Norm_SPid}, we see that
	for $\gamma \in \CcFfoPROi{U_i}{V_i}{k}{\omega}$ and $i \in I$
	\[
		\hn{\beta_i \circ (\id{U_i}, \gamma_i)}{f_i}{0}
		\leq \hn{\parFAbl{2}{\beta_i}}{1_{U_i \times V_i}}{0} \hn{\gamma_i}{f_i}{0}.
	\]
	Since $\hn{\parFAbl{2}{\beta_i}}{1_{U_i \times V_i}}{0} \leq \hn{\beta_i}{1_{U_i \times V_i}}{1}$,
	there exists $g \in \ExtWeights{\GewFunk}$ such that
	\[
		\hn{\famiI{\beta_i \circ (\id{U_i}, \gamma_i)} }{f_i}{0}
		\leq \hn{\gamma}{g_i}{0}.
	\]
	Hence
	\[
		\famiI{\beta_i \circ (\id{U_i}, \gamma_i)} \in \CcFproI{U_i}{Z_i}{0}.
	\]
	With the same reasoning, we see with \ref{est:f0-Norm_SPid-Differenz} that
	for $\eta \in \CcFfoPROi{U_i}{V_i}{k}{\omega}$ in some neighborhood of $\gamma$,
	\[
		\hn{\famiI{\beta_i \circ (\id{U_i}, \gamma_i) - \beta_i \circ (\id{U_i}, \eta_i)} }{f}{0}
		\leq \hn{\gamma - \eta}{g}{0}.
	\]
	So by \refer{lem:L-Stetigkeit_Abb_in_LinfProd-gewAbb},
	$\beta_*$ is locally Lipschitz continuous and hence continuous.

	$k \to k+1$:
	We use \refer{prop:Zerlegungssatz_Familie}.
	For $\famiI{\gamma_i} \in \CcFfoPROi{U_i}{V_i}{k}{\omega}$,
	we have by \refer{prop:SuperpostionCWZweiVars-id}
	using notation from \refer{lem:Abschaetzung_hoheDiffs_Spezialfall-linArg}
	\[
		\famiI{\FAbl{(\beta_i \circ (\id{U_i}, \gamma_i))}}
		= \famiI{\parFAbl{1}{\beta_i} \circ (\id{U_i}, \gamma_i)}
		+ \famiI{(\beta_i)^{(2)}_{M_i} \circ (\id{U_i}, \gamma_i, \FAbl{\gamma_i}) }.
	\]
	(Here, $M_i$ denotes the composition of linear operators).
	For $i \in I$ and $\ell \in \N^*$,
	\[
		\hn{\parFAbl{1}{\beta_i}}{1_{U_i \times V_i}}{\ell} \leq \hn{\beta_i}{1_{U_i \times V_i}}{\ell + 1},
	\]
	and from \ref{est:Differential-MaMu_hohes_Diff_1-l-Norm} we get that
	\[
		\hn{(\beta_i)^{(2)}_{M_i}}{1_{U_i \times V_i \times \Ball[\Lin{X_i}{Y_i}]{0}{R} }}{\ell}
				\leq \ell \hn{\beta_i}{1_{U_i \times V_i}}{\ell}
					+ R \hn{\beta_i}{1_{U_i \times V_i}}{\ell + 1}
	\]
	for each $R > 0$.
	Hence we can apply the inductive hypothesis to see that the maps
	\[
		\CcFfoPROi{U_i}{V_i}{k}{\omega}
		\to
		\CcFproI{U_i }{ \Lin{X_i}{Z_i} }{k}
		:
		\famiI{\gamma_i} \mapsto \famiI{\parFAbl{1}{\beta_i} \circ (\id{U_i}, \gamma_i)}
	\]
	and for $R \geq 1$
	\begin{equation*}
		\CcFfoPROi{U_i}{V_i \times \Ball[\Lin{X_i}{Y_i}]{0}{R}}{k}{\omega}
		\to
		\CcFproI{U_i }{ \Lin{X_i}{Z_i} }{k}
		:
		\famiI{\gamma_i} \mapsto \famiI{(\beta_i)^{(2)}_M \circ (\id{U_i}, \gamma_i) }
	\end{equation*}
	are continuous; here we used that $\omega$ is an adjusting weight for $\famiI{V_i \times \Ball[\Lin{X_i}{Y_i}]{0}{R}}$
	when the product is endowed with the maximum norm of the factor products
	(and also for $\famiI{\Ball[\Lin{X_i}{Y_i}]{0}{R}}$) if $R \geq 1$.
	From the continuity of the latter map, we deduce using \refer{lem:gewichtete_Abb_Produktisomorphie-endl},
	\refer{lem:multilineareSuperpos-Linf} and \refer{lem:pktwProduktLInf} that
	\begin{align*}
		\CcFfoPROi{U_i}{V_i}{k}{\omega} \times \CcFfoPROi{U_i}{\Ball[\Lin{X_i}{Y_i}]{0}{R}}{k}{\omega}
		&\to
		\CcFproI{U_i }{ \Lin{X_i}{Z_i} }{k}
		\\
		(\famiI{\gamma_i}, \famiI{\Gamma_i}) &\mapsto \famiI{(\beta_i)^{(2)}_M \circ (\id{U_i}, \gamma_i, \Gamma_i) }
	\end{align*}
	is continuous. Hence for each $\gamma \in \CcFfoPROi{U_i}{V_i}{k + 1}{\omega}$,
	the map
	\begin{align*}
		\set{ \eta \in  \CcFfoPROi{U_i}{V_i}{k + 1}{\omega} }{ \hn{ \eta }{1_{\disjointUI{U_i}}}{1} <  \hn{ \gamma }{1_{\disjointUI{U_i} }}{1} + 1}
		&\to \CcFproI{U_i }{ \Lin{X_i}{Z_i} }{k}
		\\
		\famiI{\eta_i} &\mapsto (\beta_i)^{(2)}_M\circ(\id{\UF}, \eta_i, \FAbl{\eta_i})
	\end{align*}
	is defined and continuous.
	In view of \refer{bem:konstantes-1-Gew_adjust-weight},
	the domain of this map is a neighborhood of $\gamma$.
	This finishes the inductive proof.

	The case $k = \infty$ follows from the case $k < \infty$
	by means of \refer{lem:CinfLinf_initial_CkLinf}.

	Now we prove that $\beta_*$ is smooth.
	More exactly, we show by induction on $\ell \in \N^*$ that it is $\ConDiff{}{}{\ell}$.

	$\ell = 1$:
	By \refer{prop:SuperpostionCWZweiVars-id},
	for any $i \in I$ the map
	\[
		(\beta_i)_* : \CFo{U_i}{V_i}{\GewFunk_i}{k} \to \CF{U_i}{Z_i}{\GewFunk_i}{k}
		: \gamma \mapsto \beta_i \circ (\id{U_i}, \gamma)
	\]
	is $\ConDiff{}{}{1}$.
	We noted in \ref{id:Differential_SuperposCWZweiVars-id} that its differential is given by
	\[
		\dA{(\beta_i)_*}{\gamma}{\eta}
		= (\dA{_{2}\beta_i}{}{})_{\ast}(\gamma, \eta).
	\]
	Obviously $\dA{_{2}\beta_i}{}{} = (\beta_i)^{(2)}_\eval$,
	where $\eval$ denotes the evaluation of linear operators.
	We see with the same reasoning as above that the map
	\begin{align*}
		\CcFfoPROi{U_i}{V_i}{k}{\omega} \times \CcFproI{U_i}{Y_i}{k}
		\to \CcFproI{U_i}{Z_i}{k}
		:
		(\gamma, \eta) \mapsto \famiI{ (\beta_i)^{(2)}_\eval)_* (\gamma_i, \eta_i) }
	\end{align*}
	is defined and continuous. Hence we can apply \refer{lem:Abb_nach_Linf_Ck_wenn_Komp_Ck_mit_stetigem_Diff}
	to see that $\beta_*$ is $\ConDiff{}{}{1}$ with $\dA{\beta_*}{}{} = \prod_{i \in I}(\dA{_2\beta_i}{}{})_*$.

	$\ell \to \ell + 1$:
	We see with the inductive hypothesis that $\prod_{i \in I}(\dA{_2\beta_i}{}{})_*$ is $\ConDiff{}{}{\ell}$,
	and since $\dA{\beta_*}{}{} = \prod_{i \in I}(\dA{_2\beta_i}{}{})_*$,
	we deduce that $\beta_*$ is $\ConDiff{}{}{\ell + 1}$.
\end{proof}
For technical reasons, we show that for a family $\famiI{\phi_i}$ of smooth maps for which
\ref{cond:est_sim-multiplier_weights} is satisfied for their Fréchet differentials $\famiI{\FAbl{\phi_i}}$,
the family of their ordinary differentials $\famiI{\dA{\phi_i}{}{}}$ satisfies \ref{cond:est_SP-Abb_weights},
at least on bounded subsets.
\begin{lem}\label{lem:vergleich_Bedingungen_simultane-multiplier_simu-Supo}
	Let $\famiI{U_i}$ be a family such that each $U_i$
	is an open nonempty set of a normed space $X_i$
	and $\famiI{Y_i}$ a family of normed spaces.
	Further, for each $i \in I$ let $\beta_i : U_i \to Y_i$ be a smooth map
	and $\GewFunk \sub \cl{\R}^{\disjointUI{U_i}}$ such that \ref{cond:est_sim-multiplier_weights}
	is satisfied for $\famiI{\FAbl{\beta_i}}$.
	Then for each $R > 0$, $\famiI{\rest{\dA{\beta_i}{}{}}{U_i \times \Ball[X_i]{0}{R}}}$ satisfies \ref{cond:est_SP-Abb_weights}.
\end{lem}
\begin{proof}
	Let $i \in I$. Then we derive from \ref{est:Abb_linear_2Arg-Spezialfall-hohes_Diff} that
	for all $\ell \in \N^*$, $x \in U_i$ and $h \in X_i$,
	\[
		\Opnorm{\FAbl[\ell]{\dA{\beta_i}{}{}}(x, h)}
		\leq \ell \Opnorm{\FAbl[\ell - 1]{\FAbl{\beta_i}}(x)} +  \norm{h}\, \Opnorm{\FAbl[\ell]{\FAbl{\beta_i}}(x)}.
	\]
	Hence
	\[
		\hn{\dA{\beta_i}{}{}}{1_{U_i \times \Ball[X_i]{0}{R}}}{\ell}
		\leq \ell \hn{\FAbl{\beta_i}}{1_{U_i}}{\ell - 1} + R \hn{\FAbl{\beta_i}}{1_{U_i}}{\ell},
	\]
	and from this estimate we easily derive that \ref{cond:est_SP-Abb_weights} is satisfied when
	\ref{cond:est_sim-multiplier_weights} is.
\end{proof}
\paragraph{Simultaneous superposition with uniformly bounded maps}
As a corollary, we prove a superposition result that is more in the style of
\cite[Prop.~3.3.12]{MR2952176};
we examine functions that are not necessarily defined on a product
and assume that the norms of the derivatives are uniformly bounded.
First, we state an obvious fact.
\begin{lem}\label{lem:glm_beschraenkte_Abb-Multiplier}
	Let $\famiI{U_i}$ and $\famiI{V_i}$ be families such that each $U_i$
	is an open nonempty subset of the normed space $X_i$
	and each $V_i$ is an open nonempty subset of a normed space $Y_i$.
	Further, let $\famiI{Z_i}$ be another family of normed spaces
	and $\GewFunk \sub \cl{\R}^{\disjointUI{U_i}}$ nonempty.
	For each $i \in I$, let $\beta_i \in \FC{ U_i \times V_i}{ Z_i}{\infty}$ be a map
	such that for each $\ell \in \N^*$,
	\[
		K_\ell \ndef \sup_{i \in I}\sset{\hn{\beta_i}{1_{U_i \times V_i}}{\ell} } < \infty.
	\]
	Then \ref{cond:est_SP-Abb_weights} is satisfied.
\end{lem}
\begin{proof}
	Let $\ell \in \N^*$. For $f \in \GewFunk$ and $i \in I$, we have that
	\[
		\hn{\beta_i}{1_{U_i \times V_i}}{\ell} \abs{f_i} \leq K_\ell \abs{f_i}.
	\]
	Since $K_\ell f \in \ExtWeights{\GewFunk}$, the assertion is proved.
\end{proof}
We now prove the result. The main difficulty is that in order to use \refer{prop:simultane_SP_BCinf0_Produkt},
we have to adapt its results for functions that are not necessarily defined on a product.
\begin{cor}\label{cor:simultane_SP_BCinf0_einfach}
	Let $\famiI{U_i}$ and $\famiI{V_i}$ be families such that each $U_i$
	is an open nonempty subset of the normed space $X_i$
	and each $V_i$ is an open subset of a normed space $Y_i$ that is star-shaped with center $0$.
	Further, let $\famiI{Z_i}$ be another family of normed spaces
	and $\GewFunk \sub \cl{\R}^{\disjointUI{U_i}}$ contain an adjusting weight $\omega$.
	For each $i \in I$, let $\beta_i \in \FC{V_i}{ Z_i}{\infty}$ be a map
	such that $\beta_i(0) = 0$.
	Further, assume that for each $\ell \in \N^*$, the set
	\[
		\set{\hn{\beta_i}{1_{V_i}}{\ell}}{i \in I}
	\]
	is bounded. Then for $k \in \cl{\N}$, the map
	\[
		\CcFfoPROi{U_i}{V_i}{k}{\omega}
		\to
		\CcFproI{U_i}{Z_i}{k}
		:
		\famiI{\gamma_i} \mapsto \famiI{\beta_i \circ \gamma_i}
	\]
	is defined and smooth.
\end{cor}
\begin{proof}
	For each $i \in I$, we define
	$\widetilde{\beta}_i : U_i \times V_i \to Z_i : (x, y) \mapsto \beta_i(y)$.
	We can calculate that
	$\FAbl[\ell]{\widetilde{\beta}_i} = \mathrm{pr}_2^* \circ (\FAbl[\ell]{\beta_i} \circ \mathrm{pr}_2)$
	(and did so in \cite[La. A.1.17]{MR2952176}),
	where $\mathrm{pr}_2 : X_i \times Y_i \to Y_i$ denotes the projection onto the second component.
	So $\hn{\widetilde{\beta}_i}{1_{U_i \times V_i}}{\ell} \leq \hn{\beta_i}{1_{V_i}}{\ell}$
	for all $\ell \in \N$.
	Further $\widetilde{\beta}_i \circ (\id{U_i}, \gamma_i) = \beta_i \circ \gamma_i$
	for each map $\gamma_i : U_i \to V_i$,
	and $\widetilde{\beta}_i(U_i \times \sset{0}) = \sset{0}$.
	Hence we derive the assertion from \refer{prop:simultane_SP_BCinf0_Produkt}
	and \refer{lem:glm_beschraenkte_Abb-Multiplier}.
\end{proof}
\paragraph{Simultaneous superposition with analytic maps}
We prove a result concerning the superposition with analytic maps.
As in \refer{cor:simultane_SP_BCinf0_einfach}, the results derived here
are in the style of \cite[Prop. 3.3.19]{MR2952176}.

We start with simultaneous \enquote{good} complexifications.
\begin{lem}\label{lem:Komplexifizierung_LinfCW_offene-Mengen}
	Let $\famiI{U_i}$ and $\famiI{V_i}$ be families such that each $U_i$
	is an open nonempty set of the normed space $X_i$,
	each $V_i$ is an open set of a real normed space $Y_i$
	and $\famiI{\widetilde{V_i}}$ a family such that for each $i \in I$,
	$\widetilde{V_i}$ is an open neighborhood of $\iota_i(V_i)$ in $(Y_i)_\C$,
	where $\iota_i : \SY_i \to (\SY_i)_\C$ denotes the canonical inclusion.
	Assume that
	\begin{equation}\label{cond:glmAbstand_Rand}
		(\forall i \in I, M \subseteq V_i)
		\,\dist{M}{Y_i\setminus V_i}
		\leq \dist{\iota_i(M)}{(Y_i)_\C\setminus\widetilde{V_i}}.
	\end{equation}
	Then
	\[
		\prod_{i \in I} (\iota_i)_*
		(\CcFoPROi{U_i}{V_i}{k})
		\sub \CcFoPROi{U_i}{\widetilde{V}_i}{k}
	\]
	for each $k \in \cl{\N}$ and $\GewFunk \sub \cl{\R}^{\disjointUI{U_i}}$ containing $1_{\disjointUI{U_i}}$.%

\end{lem}
\begin{proof}
	Note that $\prod_{i \in I} (\iota_i)_*$ is defined by \refer{lem:multilineareSuperpos-Linf}.
	Let $\gamma \in \CcFoPROi{U_i}{V_i}{k}$.
	By definition, there exists $r > 0$ such that $\gamma_i(U_i) + \Ball[Y_i]{0}{r} \sub V_i$
	for all $i \in I$; in particular, $\dist{\gamma_i(U_i)}{Y_i \setminus V_i} \geq r$.
	By \ref{cond:glmAbstand_Rand},
	$\dist{\iota_i(\gamma_i(U_i))}{(Y_i)_\C\setminus\widetilde{V_i}} \geq r$
	and hence $(\iota_i \circ \gamma_i)(U_i) + \Ball[(Y_i)_\C]{0}{r} \sub \widetilde{V_i}$
	for each $i \in I$.
	Thus
	\[
		\prod_{i \in I} (\iota_i)_*(\gamma)
		= \famiI{\iota_i \circ \gamma_i} \in \CcFoPROi{U_i}{\widetilde{V}_i}{k},
	\]
	which finishes the proof.
\end{proof}
We now prove the result. We assume that the domains of the superposition maps
do not become arbitrarily small, and that they are uniformly bounded on
subsets that have a uniform distance from the domain boundary.
This, together with the Cauchy estimates, will enable us to use \refer{prop:simultane_SP_BCinf0_Produkt}.
We need two results from \cite{arxiv_1006.5580v3} that were used in \cite{MR2952176},
but not explicitely stated.
La. 3.3.13 is a (revised) version of the approximation technique used in the proof of \cite[La. 3.3.13]{MR2952176},
and estimate (3.3.15.1) was used in the proof of \cite[La. 3.3.14]{MR2952176}.

\begin{cor}\label{cor:sim-SuperPos_analytisch}
	Let $\famiI{U_i}$ and $\famiI{V_i}$ be families such that each $U_i$
	is an open nonempty subset of a normed space $X_i$,
	each $V_i$ is an open subset of a normed space $Y_i$ that is star-shaped with center $0$
	such that $\inf_{i \in I} \dist{\sset{0}}{\partial V_i} > 0$.
	Further, let $\famiI{Z_i}$ be another family of normed spaces
	and $\GewFunk \sub \cl{\R}^{\disjointUI{U_i}}$ with $1_{\disjointUI{U_i}} \in \GewFunk$.
	For each $i \in I$, let $\beta_i : V_i \to  Z_i$ be a map with $\beta_i(0) = 0$.
	Further, assume that either all $\beta_i$ are complex analytic with
	\begin{equation}\label{cond:glm_beschr_Mengen-glm-dist-Rand}
		\bigl(\forall \famiI{W_i} : W_i \sub V_i \text{ open and bounded}, \inf_{i \in I} \dist{W_i}{\partial V_i} > 0\bigr)
		\,\sup_{i \in I} \hn{\beta_i}{ 1_{W_i} }{0} < \infty ;
	\end{equation}
	or that any $\beta_i$ is real analytic and has a complexification
	\[
		\widetilde{\beta_i} : \widetilde{V_i} \sub (Y_i)_\C \to (Z_i)_\C
	\]
	such that \ref{cond:glm_beschr_Mengen-glm-dist-Rand} is satisfied
	and whose domains $\widetilde{V_i}$ are star-shaped with center $0$
	and satisfy \ref{cond:glmAbstand_Rand}.
	Then for $k \in \cl{\N}$, the map
	\[
		\beta_\ast : \CcFoPROi{U_i}{V_i}{k}
		\to
		\CcFproI{U_i}{Z_i}{k}
		:
		\famiI{\gamma_i} \mapsto \famiI{(\beta_i)_* (\gamma_i)} = \famiI{\beta_i \circ \gamma_i}
	\]
	is defined and analytic.
\end{cor}
\begin{proof}
	We first assume that all $\beta_i$ are complex analytic.
	Let $r \in ]0, d[$, where $d \ndef \inf_{i \in I}\dist{\sset{0}}{\partial V_i}$.
	We use \cite[La. 3.3.13]{arxiv_1006.5580v3}
	to see that there
	exists a family $\famiI{\VF^{\partial,r}_i}$ such that each $\VF^{\partial,r}_i$
	is open, bounded and star-shaped with center $0$; and furthermore
	$\inf_{i \in I} \dist{\VF^{\partial,r}_i}{\partial V_i} \geq \tfrac{d - r}{2} \min(1, r^2)$
	and $\bigcup_{r < d} \VF^{\partial,r}_i = \VF_i$ for each $i \in I$.
	Hence we see with the Cauchy estimates \cite[(3.3.15.1)]{arxiv_1006.5580v3}
	that for each $\ell \in \N$, there exists $ \widetilde{r} < \tfrac{d - r}{2} \min(1, r^2)$ such that
	\[
		\hn{\beta_i}{1_{\VF^{\partial,r}_i} }{\ell}
		\leq \frac{(2 \ell)^\ell}{ (\widetilde{r})^\ell }
			\hn{\beta_i}{1_{\VF^{\partial,r}_i + \clBall[\SY_i]{0}{ \widetilde{r} } } }{0}
	\]
	for all $i \in I$.
	Using \ref{cond:glm_beschr_Mengen-glm-dist-Rand}, we conclude from this that
	\[
		\set{\hn{\beta_i}{1_{\VF^{\partial,r}_i} }{\ell}}{i \in I}
	\]
	is bounded, so we use \refer{cor:simultane_SP_BCinf0_einfach} to see that
	$\beta_*$ is defined and smooth (and hence analytic) on
	$
		\CcFoPROi{U_i}{\VF^{\partial,r}_i}{k}.
	$
	Since these sets are open in $\CcFoPROi{U_i}{V_i}{k}$ and
	\[
		\CcFoPROi{U_i}{V_i}{k}
		= \bigcup_{r \in ]0, d[}\CcFoPROi{U_i}{\VF^{\partial,r}_i}{k},
	\]
	we derive the assertion.

	Now assume that all $\beta_i$ are real analytic.
	We derive from the first part of the proof
	that $\widetilde{\beta}_* = \prod_i (\widetilde{\beta_i})_*$ is defined and analytic.
	Obviously $\beta_*$ coincides with the restriction of $\widetilde{\beta}_*$
	to $\prod_{i \in I}(\iota_i)_*(\CcFoPROi{U_i}{V_i}{k})$
	(which is contained in the domain of $\widetilde{\beta}_*$ by \refer{lem:Komplexifizierung_LinfCW_offene-Mengen}),
	hence $\beta_*$ is real analytic.
\end{proof}
We provide an application.
\begin{lem}\label{lem:sim-SuperPos_QuasiInversion}
	Let $\famiI{U_i}$ be a family such that
	each $U_i$ is an open nonempty subset of the normed space $X_i$,
	$\famiI{Y_i}$ a family of Banach spaces,
	$\GewFunk \sub \cl{\R}^{\disjointUI{U_i}}$ with $1_{\disjointUI{U_i}} \in \GewFunk$
	and $k \in \N$.
	Then the map
	\[
		\CcFoPROi{U_i}{\Ball[\Lin{Y_i}{Y_i}]{0}{1} }{k}\to \CcFproI{U_i }{ \Lin{Y_i}{Y_i} }{k}
		: \gamma \mapsto \famiI{\QuasiInv_{\Lin{Y_i}{Y_i}} \circ \gamma_i }
	\]
	is defined and analytic.
\end{lem}
\begin{proof}
	This is simply an application of \refer{cor:sim-SuperPos_analytisch}
	since each $\rest{\QuasiInv_{\Lin{Y_i}{Y_i}}}{\Ball[\Lin{Y_i}{Y_i}]{0}{1}}$
	can be written as a (the same) power series, and hence satisfies
	\ref{cond:glm_beschr_Mengen-glm-dist-Rand}.
\end{proof}
\subsection{Simultaneous composition and inversion}
\label{susec:Simultaneous_Compo_Inv}
We examine the simultaneous application of the composition and inversion operations, respectively,
that we stated in \refer{prop:Kompo_Koord_glatt}
and \refer{prop:Zsf_Inversion_gewAbb}.
\paragraph{Simultaneous composition}
We start with composition. Note that we need the adjusting weight $\omega$ to ensure that
$\CcFfoPROi{U_i}{V_i}{k}{\omega}$ is open and not empty.
\begin{prop}\label{prop:Simultane_Koor-Kompo_diffbar}
	Let $\famiI{U_i}$, $\famiI{V_i}$ and $\famiI{W_i}$ be families such that
	for each $i \in I$, $U_i$, $V_i$ and $W_i$ are open nonempty sets of the normed space $X_i$
	with $U_i + V_i \sub W_i$, and $V_i$ is balanced.
	Further, let $\famiI{Y_i}$ be another family of normed spaces
	and $\GewFunk \sub \cl{\R}^{\disjointUI{W_i}}$ contain an adjusting weight $\omega$ for $\famiI{V_i}$.
	Then for $k, \ell \in \cl{\N}$, the map
	\begin{equation*}
		\compIdDiffKLcW{Y}{k}{\ell}
		\ndef \prod_{i \in I} \compIdDiffKLWeights{Y_i}{k}{\ell}{\GewFunk_i}
		:\left\lbrace
		\begin{aligned}
			\CcFproI{W_i}{Y_i}{k + \ell + 1} \times \CcFfoPROi{U_i}{V_i}{k}{\omega}
			&\to \CcFproI{U_i}{Y_i}{k}
			\\
			(\famiI{\gamma_i}, \famiI{\eta_i}) &\mapsto \famiI{\gamma_i \circ (\eta_i + \id{U_i})}
		\end{aligned}\right.
	\end{equation*}
	is defined and $\ConDiff{}{}{\ell}$.
\end{prop}
\begin{proof}
	We see with \refer{prop:Kompo_Koord_glatt}
	(and \refer{bem:CcFfoPRO_pktw_in_CcFo})
	that $\compIdDiffKLWeights{Y}{k}{\ell}{\GewFunk}$ is defined as a map to $\prod_{i \in I} \CcF{U_i}{Y_i}{k}$.
	We first prove by induction on $k$ that $\compIdDiffKLWeights{Y}{k}{0}{\GewFunk}$ takes its values in $\CcFproI{U_i}{Y_i}{k}$
	and is continuous.

	$k = 0$:
	We see with \refer{est:Funktionswerte_Gewicht_K-Kompo}
	that for $f \in \GewFunk$, $\gamma \in \CcFproI{W_i}{Y_i}{1}$ and $\eta \in \CcFfoPROi{U_i}{V_i}{0}{\omega}$
	\[
		\hn{\compIdDiffKLWeights{Y_i}{0}{0}{\GewFunk_i}(\gamma_i, \eta_i)}{f_i}{0}
		\leq \hn{\gamma_i}{1_{U_i}}{1}\hn{\eta_i}{f_i}{0} + \hn{\gamma_i}{f_i}{0}
	\]
	for each $i \in I$. So $\compIdDiffKLcW{Y}{0}{0}$ is defined,
	taking \refer{bem:konstantes-1-Gew_adjust-weight} into account.
	Further, we see with the same reasoning – applied to \refer{est:f,0-Norm_Differenz_Kompo} –
	and \refer{lem:L-Stetigkeit_Abb_in_LinfProd-gewAbb} that
	$\compIdDiffKLcW{Y}{0}{0}$ is locally Lipschitz continuous and hence continuous.

	$k \to k + 1$:
	We use \refer{prop:Zerlegungssatz_Familie}.
	For $\gamma \in \CcFproI{W_i}{Y_i}{k + 2}$
	and $\eta \in \CcFfoPROi{U_i}{V_i}{k + 1}{\omega}$,
	for each $i \in I$ we have
	\[
		\FAbl{(\gamma_i \circ (\eta_i + \id{U_i}))}
		= \FAbl{\gamma_i} \circ (\eta_i + \id{U_i}) \MaMu (\FAbl{\eta_i} + \idco)
		= \compIdDiffKLWeights{\Lin{X_i}{Y_i}}{k}{0}{\GewFunk_i}(\FAbl{\gamma_i}, \eta_i) \MaMu (\FAbl{\eta_i} + \idco).
	\]
	By the inductive hypothesis, the map $\compIdDiffKLcW{\Lin{X}{Y}}{k}{0}$
	is defined and continuous. Further, we see (noting \refer{bem:konstantes-1-Gew_adjust-weight})
	that $\famiI{\FAbl{\eta_i} + \idco} \in \CFproI{U_i}{\Lin{X_i}{X_i}}{\sset{1_{\disjointUI{U_i}}}}{k}$.
	Hence we can apply \refer{lem:multilineareSuperpos-Linf}
	to finish the proof.

	The case $k = \infty$ follows from the case $k < \infty$
	using \refer{lem:CinfLinf_initial_CkLinf}.

	Now we prove by induction on $\ell \in \N^*$ that $\compIdDiffKLcW{Y}{k}{\ell}$ is $\ConDiff{}{}{\ell}$.

	$\ell = 1$:
	We know from \refer{prop:Kompo_Koord_glatt} that
	\[
		\compIdDiffKLWeights{Y_i}{k}{1}{\GewFunk_i} :
		\CF{W_i}{Y_i}{\GewFunk_i}{k + 2} \times \CFo{U_i}{V_i}{\GewFunk_i}{k} \to \CF{U_i}{Y_i}{\GewFunk_i}{k}
		: (\gamma, \eta) \mapsto \gamma \circ (\eta + \id{U_i})
	\]
	is $\ConDiff{}{}{1}$ for each $i \in I$, and
	we noted in \refer{id:Ableitung_Kompo} that its differential is given by
	\[
		\dA{\,\compIdDiffKLWeights{Y_i}{k}{1}{\GewFunk_i} }{\gamma, \eta}{\gamma_1, \eta_1}
		=
		\compIdDiffKLWeights{\Lin{X_i}{Y_i}}{k}{0}{\GewFunk_i}(\FAbl{\gamma}, \eta) \eval \eta_1
		+ \compIdDiffKLWeights{Y_i}{k}{1}{\GewFunk_i}(\gamma_1, \eta).
	\]
	Since we already proved that $\compIdDiffKLcW{\Lin{X}{Y} }{k}{0}$
	and $\compIdDiffKLcW{Y}{k}{1}$ are continuous,
	we use \refer{lem:multilineareSuperpos-Linf} to see that
	\begin{gather*}
		\CcFproI{W_i}{Y_i}{k + \ell + 1} \times \CcFfoPROi{U_i}{V_i}{k}{\omega}
		\times \CcFproI{W_i}{Y_i}{k + \ell + 1} \times \CcFproI{U_i}{X_i}{k}
		\to \CcFproI{U_i}{Y_i}{k}
		\\
		(\gamma, \eta, \gamma^1, \eta^1)
		\mapsto
		\famiI{\compIdDiffKLWeights{\Lin{X_i}{Y_i}}{k}{\ell - 1}{\GewFunk_i}(\FAbl{\gamma_i}, \eta_i) \eval \eta_i^1
				+ \compIdDiffKLWeights{Y_i}{k}{\ell}{\GewFunk_i}(\gamma_i^1, \eta_i)}
	\end{gather*}
	is defined and continuous.
	Hence we can apply \refer{lem:Abb_nach_Linf_Ck_wenn_Komp_Ck_mit_stetigem_Diff}
	to see that $\compIdDiffKLcW{Y}{k}{\ell}$ is $\ConDiff{}{}{1}$
	and $\dA{\compIdDiffKLcW{Y}{k}{\ell}}{}{}$ is given by this map.

	$\ell \to \ell + 1$:
	We apply the inductive hypothesis and \refer{lem:multilineareSuperpos-Linf}
	to the identity for $\dA{\compIdDiffKLcW{Y}{k}{\ell + 1}}{}{}$ derived above
	to see that $\dA{\compIdDiffKLcW{Y}{k}{\ell + 1}}{}{}$ is $\ConDiff{}{}{\ell}$,
	hence $\compIdDiffKLcW{Y}{k}{\ell + 1}$ is $\ConDiff{}{}{\ell + 1}$.
\end{proof}

\paragraph{Simultaneous inversion}
We treat inversion. Here an adjusting weight is given explicitly.
\begin{prop}\label{prop:Simultane_Inv-Kompo_glatt}
	Let $\famiI{U_i}$ and $\famiI{\widetilde{U}_i}$ be families
	such that $U_i$ and $\widetilde{U}_i$ are open nonempty sets of the Banach space $X_i$
	and each $U_i$ is convex.
	Further assume that there exists $r > 0$ such that $\widetilde{U}_i + \Ball[X_i]{0}{r} \sub U_i$
	for all $i \in I$.
	Let $\GewFunk \sub \cl{\R}^{\disjointUI{U_i}}$ with $1_{\disjointUI{U_i}} \in \GewFunk$
	and $\tau \in ]0, 1[$.
	Then the map
	\[
		\InvIdcW{\widetilde{U}}
		\ndef
		\prod_{i \in I}
		\InvIdWeights{\widetilde{U}_i}{\GewFunk_i}
		:
		\mathcal{D}^\tau
		\to
		\CcFproI{\widetilde{U}_i}{X_i}{\infty}
		:
		\famiI{\phi_i} \mapsto \famiI{ \rest{(\phi_i + \id{U_i})^{-1}}{\widetilde{U}_i} - \id{\widetilde{U}_i} }
	\]
	is defined and smooth, where
	\[
		\mathcal{D}^\tau \ndef \left\set{\phi \in \CcFproI{U_i}{X_i}{\infty} }{%
			\hn{\phi}{1_{\disjointUI{U_i}}}{1} < \tau
			\text{ and }
			\hn{\phi}{1_{\disjointUI{U_i}}}{0}
			< \tfrac{r}{2} (1 - \tau)
		\right}.
	\]
\end{prop}
\begin{proof}
	We use \refer{prop:Zsf_Inversion_gewAbb} to see that $\InvIdcW{\widetilde{U}}$
	is defined as a map to $\prod_{i \in I} \CcFproI{\widetilde{U}_i}{X_i}{\infty}$.
	We prove by induction on $k$ that it takes values in $\CcFproI{\widetilde{U}_i}{X_i}{k}$
	and is continuous.

	$k = 0$:
	By \refer{est:Abschaetzung_gewichteter_FWert_der_K-Inversion},
	we have for $f\in\GewFunk$, $\famiI{\phi_i} \in \mathcal{D}^\tau$ and each $i \in I$ that
	\[
		\hn{\InvIdWeights{\widetilde{U}_i}{\GewFunk_i}(\phi_i) }{f_i}{0}
		\leq \hn{\phi_i}{f_i}{0} \tfrac{1}{1 - \hn{\phi_i}{1_{\widetilde{U}_i}}{1}}
		\leq \tfrac{1}{1 - \tau} \hn{\phi_i}{f_i}{0}.
	\]
	Since $\tau < 1$ and $i$ was arbitrary, $\InvIdcW{\widetilde{U}}$ is defined.
	In the same manner, we can use \refer{est:f0-norm_Diff_KoorInv}
	to see with \refer{lem:L-Stetigkeit_Abb_in_LinfProd-gewAbb} that $\InvIdcW{\widetilde{U}}$
	is locally Lipschitz continuous and hence continuous.

	$k \to k + 1$:
	We use \refer{prop:Zerlegungssatz_Familie}.
	By \refer{id:Differential_der_inversen_Abb},
	for $\phi \in \mathcal{D}^\tau$,
	\[
		\famiI{\FAbl{\, \InvIdWeights{\widetilde{U}_i}{\GewFunk_i}(\phi_i)} }
		= \famiI{\compIdWeights{\Lin{\SX_i}{\SX_i}}{\GewFunk_i}(\FAbl{\phi_i} \MaMu \QuasiInv(- \FAbl{\phi_i}) - \FAbl{\phi_i},
				\InvIdWeights{\widetilde{U}_i}{\GewFunk_i}(\phi_i))}.
	\]
	Since $\famiI{\FAbl{\phi_i}} \in \CcFoPROi{U_i}{\Ball[\Lin{X_i}{X_i}]{0}{1} }{k}$,
	we can apply \refer{lem:sim-SuperPos_QuasiInversion} and after that \refer{lem:multilineareSuperpos-Linf},
	\refer{prop:Simultane_Koor-Kompo_diffbar} and the inductive hypothesis
	to finish the proof.

	The case $k = \infty$ follows from the case $k < \infty$
	with \refer{lem:CinfLinf_initial_CkLinf}.

	Now we prove that $\InvIdcW{\widetilde{U}}$ is smooth.
	More exactly, we show by induction on $\ell \in \N^*$ that it is $\ConDiff{}{}{\ell}$.

	$\ell = 1$:
	By \refer{prop:Zsf_Inversion_gewAbb},
	the map $\InvIdWeights{\widetilde{U}_i}{\GewFunk_i}$
	is $\ConDiff{}{}{1}$ on $\pi_i(\mathcal{D}^\tau)$ for each $i \in I$,
	and we stated in \refer{id:Ableitung_Inversion}
	that its differential is
	\[
		\dA{ \,\InvIdWeights{\widetilde{U}_i}{\GewFunk_i}}{\phi}{\phi^1}
			= \compIdWeights{\SX_i}{\GewFunk_i}(\QuasiInv(\FAbl{\phi}) \eval \phi^1 + \phi^1,
				\InvIdWeights{\widetilde{U}_i}{\GewFunk_i}(\phi) ).
	\]
	We conclude using \refer{lem:sim-SuperPos_QuasiInversion}, \refer{lem:multilineareSuperpos-Linf}
	\refer{prop:Simultane_Koor-Kompo_diffbar} and the continuity of $\InvIdcW{\widetilde{U}}$
	that the map
	\begin{align*}
		\mathcal{D}^\tau \times \CcFproI{U_i}{X_i}{\infty}
		\to
		\CcFproI{\widetilde{U}_i}{X_i}{\infty}
		:
		(\phi, \phi^1)
		\mapsto
		\famiI{ \compIdcW{\SX_i}(\QuasiInv(\FAbl{\phi_i}) \eval \phi^1_i + \phi^1_i,
		\InvIdWeights{\widetilde{U}_i}{\GewFunk_i}(\phi_i) ) }
	\end{align*}
	is continuous.
	So we can apply \refer{lem:Abb_nach_Linf_Ck_wenn_Komp_Ck_mit_stetigem_Diff}
	to see that $\InvIdcW{\widetilde{U}}$ is $\ConDiff{}{}{1}$ and its differential is given by this map.

	$\ell \to \ell + 1$:
	We apply the inductive hypothesis, \refer{lem:sim-SuperPos_QuasiInversion},
	\refer{lem:multilineareSuperpos-Linf} and \refer{prop:Simultane_Koor-Kompo_diffbar}
	to the identity for $\dA{\InvIdcW{\widetilde{U}}}{}{}$ derived above
	to see that $\dA{\InvIdcW{\widetilde{U}}}{}{}$ is $\ConDiff{}{}{\ell}$,
	hence $\InvIdcW{\widetilde{U}}$ is $\ConDiff{}{}{\ell + 1}$.
\end{proof}

\begin{bem}
	We implicitly used in this subsection that the operator norms of the composition resp. evaluation
	of linear maps are uniformly bounded.
\end{bem}

\printbibliography
\end{document}